\documentclass[11pt]{article}
\usepackage{multicol}
\usepackage{amssymb,amsmath,amsthm,bm}
\usepackage[colorlinks=true, urlcolor=blue,linkcolor=blue, citecolor=blue]{hyperref}
\usepackage{mathrsfs,verbatim}
\usepackage{caption,cleveref }
\usepackage{graphicx, float}
\usepackage{color, mathtools, tikz, xcolor}
\usepackage{enumerate,enumitem}
\usepackage{tikz}
\usetikzlibrary{shapes.geometric}
\usepackage{bbm}

\usetikzlibrary{calc}

\usepackage[mathscr]{euscript}
\usepackage{appendix}
\usepackage{geometry}
\numberwithin{equation}{section}

	\newtheorem{theorem}{Theorem}[section]

	\newtheorem{corollary}[theorem]{Corollary}
	\newtheorem{conjecture}[theorem]{Conjecture}
	\newtheorem{question}[theorem]{Question}
	\newtheorem{claim}[theorem]{Claim}
	\newtheorem{proposition}[theorem]{Proposition}
	\newtheorem{lemma}[theorem]{Lemma}
	\newtheorem{fact}[theorem]{Fact}

\newenvironment{proofclaim}[1][Proof of claim]{\begin{proof}[#1]}{\end{proof}}

\def\al#1{}
	\renewcommand{\al}[1]{\footnote{\textbf{AL: }#1}} 

\def\eps{{\varepsilon}}

\title{Towards an edge-coloured Corr\'adi--Hajnal theorem}

\author{ Allan Lo\thanks{School of Mathematics, University of Birmingham, Birmingham, B15 2TT, UK. \texttt{s.a.lo@bham.ac.uk}. The research leading to these results was supported by EPSRC, grant no. EP/V002279/1 and EP/V048287/1. There are no additional data beyond that contained within the main manuscript.} \and Ella Williams\thanks{Department of Mathematics, University College London, Gower Street, WC1E 6BT, UK. \texttt{ella.williams.23@ucl.ac.uk}. Research supported by the Martingale Foundation.}}
\date{\today}

\begin{document}

\maketitle

\begin{abstract}
A classical result of Corr\'adi and Hajnal states that every graph~$G$ on $n$ vertices with $n\in~3\mathbb{N}$ and $\delta(G) \ge 2n/3$ contains a perfect triangle-tiling, i.e.,\ a spanning set of vertex-disjoint triangles. 
We explore a generalisation of this result to edge-coloured graphs. 
Let $G$ be an edge-coloured graph on $n$ vertices.
The minimum colour degree~$\delta^c(G)$ of~$G$ is the largest integer~$k$ such that, for every vertex~$v \in V(G)$, there are at least $k$ distinct colours on edges incident to~$v$.
We show that if $\delta^c(G) \ge (5/6 + \varepsilon) n$, then $G$ has a spanning set of vertex-disjoint rainbow triangles. 
On the other hand, we find an example showing the bound should be at least~$5n/7$.
We also discuss a related tiling problems on digraphs, which may be of independent interest.
\end{abstract}

\section{Introduction}

Given graphs $G$ and $H$, a common problem in extremal graph theory is to determine how large the minimum degree of $G$, denoted by $\delta(G)$, must be in order to guarantee that $G$ contains a copy of $H$ as a subgraph. 
In this paper, we consider a generalisation of this type of problem to edge-coloured graphs. An \emph{edge-coloured graph} is a graph $G$ together with an edge-colouring of $G$. We say that $G$ is \emph{properly coloured} if no two adjacent edges in $G$ have the same colour. If all edges have a distinct colour, then $G$ is \emph{rainbow}.
We are interested in finding rainbow subgraphs in edge-coloured graphs. Keevash, Mubayi, Sudakov and Verstra\"{e}te~\cite{KeevashMubayiSudakovVerstraete2007} studied this from a Tur\'{a}n-type perspective, considering how many edges a host properly edge-coloured graph $G$ on $n$ vertices must have to guarantee the existence of a specified rainbow subgraph. In particular, they proved that if $H$ is a non-bipartite graph and $G$ is properly edge-coloured with $e(G) \ge (1+o(1)){\rm ex} (n, H)$, then $G$ contains a rainbow~$H$, where the Tur\'{a}n number ${\rm ex} (n, H)$ denotes the maximum number of edges in an $H$-free $n$-vertex graph.

We instead consider a generalised notion of minimum degree for edge-coloured graphs given as follows. 
Let $G$ be an edge-coloured graph on $n$ vertices. 
Given a vertex $v \in V(G)$, the \emph{colour degree}~$d^c(v)$ of~$v$ is the number of distinct colours on edges incident to~$v$.
The \emph{minimum colour degree}~$\delta^c(G)$ is the minimum $d^c(v)$ over all vertices~$v$ in~$G$. 

Li~\cite{Li2013} proved that if $\delta^c(G) > n/2$, then $G$ contains a rainbow~$K_3$. 
Later, Li, Ning, Xu and Zhang~\cite{LiNingXuZhang2014} showed that, when $n \ge 5$, the properly coloured balanced complete bipartite graph is the unique extremal example. 

\begin{theorem}[Li, Ning, Xu and Zhang~\cite{LiNingXuZhang2014}] \label{thm:rainbowK_3}
Let $G$ be an edge-coloured graph on $n \geq 3$ vertices with $\delta^c(G) \ge n/2$. 
Suppose that $G$ is not a $K_{n/2,n/2}$, $K_4$ or $K_4 - e$. 
Then $G$ contains a rainbow~$K_3$. 
\end{theorem}
This result was further generalised by Czygrinow, Molla and Nagle~\cite{CzygrinowMollaNagle}, who showed that if $\delta^c(G) > (1- 1/r)n$, then $G$ contains a rainbow~$K_{r+1}$. Together with Oursler~\cite{Czygrinow2019even,CZYGRINOW2021odd}, they also proved that for a fixed $\ell$ and sufficiently large $n$, if $\delta^c(G)\geq (n+1)/2$ then $G$ contains a rainbow cycle of length $\ell$, and that this bound can be improved to $\delta^c(G)\geq(n+5)/3$ when $\ell$ is even. 
We give an alternative proof of Theorem~\ref{thm:rainbowK_3} in Section~\ref{sec:rainbowK_3}.

For a graph~$H$, an \emph{$H$-tiling} is a collection of vertex-disjoint copies of~$H$.  
An $H$-tiling is \emph{perfect} if it is spanning. 
A classical result of Corr\'adi and Hajnal~\cite{CorradiHajnal1963} states that if $G$ is a graph on $n$ vertices with $n \equiv 0 \pmod{3}$ and $\delta(G) \ge 2n/3$, then $G$ contains a perfect $K_3$-tiling. 
Hajnal and Szemer\'edi~\cite{HajnalSzemeredi1970} generalised this result for perfect $K_r$-tilings. 
In this paper, we consider the following variation of $H$-tilings for edge-coloured graphs. 
A \emph{rainbow-$H$-tiling} is a collection of vertex-disjoint rainbow~$H$.
    Note that colour may repeat between different copies of~$H$. 
A rainbow-$K_3$-tiling is a properly coloured $K_3$-tiling.

For an edge-coloured graph $G$ on $n$ vertices,
Hu, Li and Yang~\cite{HuLiYang2020} proved that if $\delta^c(G) \ge (n+2)/2$ and $n \ge 20$, then $G$ contains two vertex-disjoint rainbow~$K_3$.
This was extended to all $n \ge 6$ by Chen, Li and Ning~\cite{ChenLiNing2022}.
Hu, Li and Yang~\cite{HuLiYang2020} conjectured the following on the minimum colour degree needed for a rainbow-$K_3$-tiling.

\begin{conjecture}[Hu, Li and Yang~\cite{HuLiYang2020}]
\label{conj:HLY}
Let $G$ be an edge-coloured graph on $n$ vertices with $\delta^c(G) \ge (n+k)/2$.
Then $G$ contains a rainbow-$K_3$-tiling of size~$k$. 
\end{conjecture}

We give a construction showing that this conjecture is false when $k \ge 5n/17$.

\begin{proposition} \label{prop:K_3construction}
Let $n,k \in \mathbb{N}$ with $n/9 \le k \le n/3$ and $n \equiv 2k \pmod{7}$. 
Then there exists an edge-coloured graph $G$ on $n$ vertices with $\delta^c(G) = (n+12k)/7-1$,
 which does not contain a rainbow-$K_3$-tiling of size~$k$. 
\end{proposition}

In particular, when $k = n/3$, there exists an edge-coloured graph~$G$ on $n$ vertices with $\delta^c(G)=5n/7-1$ without a perfect rainbow-$K_3$-tiling.
We generalise the construction for perfect rainbow-$K_r$-tilings.

\begin{proposition} \label{prop:extremalconstruction}
Let $r,m \in \mathbb{N}$ with $m \equiv 0 \pmod{r}$. 
Then there exists an edge-coloured graph $G$ on~$(r^2 -2)m$ vertices with $\delta^c(G) = (r^2-r-1)m-1 $, which does not contain a perfect rainbow-$K_r$-tiling. 
\end{proposition}

On the other hand, we give an upper bound on the minimum colour degree threshold for the existence of perfect rainbow-$K_r$-tilings.

\begin{theorem} \label{thm:main}
For all $r \in \mathbb{N}$ with $r \ge 3$ and $\eps > 0$, there exists an integer $n_0 = n_0(r,\eps)$ such that every edge-coloured graph $G$ on $n \ge n_0$ vertices with $n \equiv 0 \pmod{r}$ and 
\begin{align*}
\delta^c(G) \ge 
	\begin{cases}
	(5/6 + \varepsilon) n & \text{if $r = 3$,}\\
	(2+\sqrt{13/2}+ \eps)n/5 \approx(0.9099+\eps)n & \text{if $r =4$,}\\
	\left(1- \frac{1}{(2r-3)r}+ \eps \right) n & \text{if $r \ge 5$,}
	\end{cases}
\end{align*}
 contains a perfect rainbow-$K_r$-tiling. 
\end{theorem}

\subsection{Organisation of the paper and proof sketch of Theorem~\ref{thm:main}}
In the next section, we set up basic notation and prove Propositions~\ref{prop:K_3construction} and~\ref{prop:extremalconstruction}. 
In Section~\ref{sec:rainbowK_3}, we give an alternative proof of Theorem~\ref{thm:rainbowK_3}. 

The rest of the paper will focus on the proof of Theorem~\ref{thm:main}.
Our approach is based on the absorption technique, introduced by R\"odl, Ruci\'nski and Szemer\'edi~\cite{RodlRucinskiSzemeredi2009}.
This reduces our problem to finding an absorbing set and an almost perfect rainbow-$K_r$-tiling. 
Finding an absorbing set is relatively straightforward and is proved in Section~\ref{sec:absorption}.

In order to find an almost perfect rainbow-$K_r$-tiling, one key ingredient in our proof is to convert edge-coloured graphs to digraphs, originating from Czygrinow, Molla and Nagle~\cite{CzygrinowMollaNagle}, as seen in Section~\ref{sec:convert}.
This allows us to apply the digraph version of the Szemer\'edi regularity lemma (Lemma~\ref{lma:szemeredi}).

In Section~\ref{sec:almosttiling} we find an almost perfect rainbow-$K_r$-tiling. 
We first convert the problem into a digraph tiling problem. 
We then consider its fractional relaxation and apply Farkas' Lemma.
Theorem~\ref{thm:main} is proved in Section~\ref{sec:proofmain}.

In Section~\ref{sec:remark}, we give remarks on  perfect rainbow-$K_r$-tiling  and related problems, and in Appendix~\ref{sec:absorbingrevisit} we prove a stronger version of the absorption lemma specifically for rainbow-$K_3$-tilings.

\section{Notation and extremal constructions} \label{sec:notation}

Let $G$ be a graph. 
For a vertex subset $S \subseteq V(G)$, we denote by $G[S]$ the induced subgraph of~$G$ on~$S$. 
For $S,T \subseteq V(G)$, we denote by $G[S,T]$ the subgraph of~$G$ with vertex set~$S \cup T$
and edge set $E(G[S,T]) = \{xy \in E(G) : x \in S \text{ and } y \in T \}$.
Note that $G[S,S] = G[S]$.
For $x,y,z\in V(G)$, we write $G[xyz]$ instead of $G[\{x,y,z\}]$.
For $t \in \mathbb{N}$, the \emph{$t$-blowup of~$G$}, denoted by~$G(t)$, is the graph such that $V(G(t))$ is obtained by replacing each $v \in V (G)$ by a vertex set $U_v = \{v_1, \dots, v_t\}$ of size~$t$ and~$E(G(t) ) = \{u_iv_j : uv \in E(G) \text{ and } i,j \in [t] \}$.  For graphs $G$ and $H$, we denote by~$G - H$ the subgraph of $G$ obtained by removing the edges in~$G \cap H$, that is $V(G-H) = V(G)$ and $E(G-H) = E(G) \setminus E(H)$.

These notation are also generalised to edge-coloured graphs and digraphs. 

An \emph{edge-coloured} graph is a graph~$G$, together with an edge-colouring.
Throughout the paper, we will assume that its edge-colouring is~$c$.
Thus, for an edge $e \in E(G)$, $c(e)$ is the colour of~$e$. 
We say that $G$ is \emph{edge-coloured critical} if any deletion of edges results in a decrease in~$\delta^c(G)$. 
Note that if $G$ is \emph{edge-coloured critical}, then each colour induces vertex-disjoint stars.

Let $G$ be a digraph on~$V$. 
For $U,W \subseteq V$, we write $G[U,W]$ to be the bipartite graph with vertex classes~$U$ and~$W$ such that $uw \in E(G[U,W])$ with $ u \in U$ and $w \in W$ if and only if $uw$ is a directed edge in~$G$.
We write $e_G(U,W)$ for $|E(G[U,W])|$.
For a digraph~$G$, we define $G^{\pm}$ to be the graph on $V(G)$ such that $xy \in E(G^{\pm})$ if and only if $xy, yx \in E(G)$.

For $k \in \mathbb{N}$, we sometimes denote the set~$\{1, 2, . . . , k\}$ by~$[k]$.
For $\alpha, \beta \in (0,1]$, we often use the notation $\alpha \ll \beta$ to mean that $\alpha$ is sufficiently small as a function of~$\beta$ i.e.\ $\alpha \leq f(\beta)$ for some implicitly given non-decreasing function $f:(0,1] \rightarrow (0,1]$. 
We implicitly assume all constants in such hierarchies are positive, and if $1/k$ appears we
assume $k$ is an integer.
We omit floors and ceilings whenever this does not affect the argument.

\subsection{Lower bound constructions}

We now give the edge-coloured graphs that prove Propositions~\ref{prop:K_3construction} and~\ref{prop:extremalconstruction}.

\begin{proof}[Proof of Proposition~\ref{prop:K_3construction}]
Let $X, Y,Z$ be disjoint vertex sets of sizes $6(n-2k)/7$, $(18k-2n)/7-1$ and $3(n-2k)/7+1$, respectively.
Define $G$ to be the edge-coloured graph on $X \cup Y\cup Z$ such that 
\begin{itemize}
	\item $G[X]$ is a rainbow complete bipartite graph with each vertex class of size~$3(n-2k)/7$;
	\item $G[Y \cup Z]$ is a rainbow complete graph;
	\item for each $x \in X$ and $y \in Y$, there is an edge~$xy$ of colour~$y$ not appearing in $G[X] \cup G[Y \cup Z]$. 
\end{itemize}
Note that $\delta^c(G) = \min \{ |X|/2 + |Y|, |Y \cup Z|, |Y \cup Z| -1 \} =  (n+12k)/7-1$. 
Every rainbow $K_3$ that contains a vertex of $X$ must also contain two vertices of~$Y$. 
Consider any rainbow-$K_3$-tiling.
The number of vertices of $X$ that are not covered is at least $|X| - \lfloor |Y|/2 \rfloor = n - 3k+1$.
Thus the rainbow-$K_3$-tiling has size at most $k-1$.
\end{proof}

\begin{proof}[Proof of Proposition~\ref{prop:extremalconstruction}]
Let $X, Y,Z$ be disjoint vertex sets of sizes $(r-1)m, (r-1)^2m-1$ and $(r-2)m+1$, respectively.
Define $G$ to be the edge-coloured graph on $X \cup Y\cup Z$ such that 
\begin{itemize}
	\item $G[X]$ is a rainbow complete $(r-1)$-partite graph with each vertex class of size $m$;
	\item $G[Y \cup Z]$ is a rainbow $K_{(r^2-r-1)m}$;
	\item For each $x \in X$ and $y \in Y$, there is an edge~$xy$ of colour~$y$ not appearing in $G[X] \cup G[Y \cup Z]$. 
\end{itemize}
Note that $\delta^c(G) = \min \{ (r-2)m + |Y|, |Y \cup Z|, |Y \cup Z| -1 \} =  (r^2-r-1)m-1$. 
Every rainbow~$K_r$ that contains a vertex of $X$ must also contain $r-1$ vertices of~$Y$. 
Since $|Y| < (r-1) |X|$, $G$ does not contain a perfect rainbow-$K_r$-tiling.
\end{proof}


\section{An alternative proof of Theorem~\ref{thm:rainbowK_3}} \label{sec:rainbowK_3}

Before providing the proof, we will need the following notation that is only used in this section. 
Let $G$ be an edge-coloured graph. 
For an edge $xy \in E(G)$, let $d_{c(xy)}(x)$ be the number of edges incident to~$x$ of colour~$c(xy)$.

\begin{proof}[Proof of Theorem~\ref{thm:rainbowK_3}]
If $n =3$, then $\delta^c(G) \ge 2$ and so $G$ is a rainbow~$K_3$. If $n=4$, there are no edge-coloured graphs satisfying the assumption to consider.
Thus we may assume that $n \ge 5$. 

Suppose to the contrary that $G$ does not contain any rainbow~$K_3$.
Let $G^*$ be an edge-coloured critical spanning subgraph of~$G$ with $\delta^c(G^*) = \delta^c(G)$.
We now claim that it suffices to show that $n$ is even and $G^*$ is a properly coloured $K_{n/2,n/2}$.
Indeed, if there is an edge~$xy \in E(G - G^*)$, then $xy$ lies entirely in one of the vertex classes of~$G^*$. 
Since $n/2 \ge 3$ and $G^* = K_{n/2,n/2}$ is properly coloured, there exists a vertex~$z$ in the other vertex class such that $c(xz) \ne c(xy) \ne c(yz)$. 
Note that $c(xz) \ne c(yz)$ as $G^*$ is properly coloured and so $xyz$ is a rainbow~$K_3$.

Hence, we may assume that $G$ is edge-coloured critical but is not a properly coloured~$K_{n/2,n/2}$.
Each monochromatic subgraph of $G$ is a disjoint union of stars.
Without loss of generality, we may further assume that each monochromatic subgraph is a star.
Thus, for $xy \in E(G)$, 
\begin{align}
  d_{c(xy)}(x)  + d_{c(xy)}(y) -2 = & \max \{d_{c(xy)}(x), d_{c(xy)}(y)\} - 1 
 \le \Delta(G) - \delta^c(G). 
	\label{eqn:star}
\end{align}
Let $x \in V(G)$ with $d(x) = \Delta(G)$.
Define a digraph $H$ on $N(x)$ such that $yz \in E(H)$ if and only if 
$c(yz) \ne c(xy) \ne c(xz)$.
For $y \in V(H)$, we have
\begin{align}
	d^+(y) & \ge ( d(x) - d_{c(xy)}(x) ) + (d(y) - d_{c(xy)}(y)) - (n-2) \nonumber
	\\
	&= 
	d(x) + d(y) - n - \left( d_{c(xy)}(x) + d_{c(xy)}(y) -2 \right) \nonumber\\
	& \overset{\mathclap{\text{\eqref{eqn:star}}}}{\ge} d(x) + d(y) - n -  ( \Delta(G) - \delta^c(G))
	= d(y) + \delta^c(G)- n. \label{eqn:d+(y)}
\end{align}
Observe that if $ yz \in E(H)$ then $c(yz) = c(xz)$ or else $xyz$ is a rainbow~$K_3$ in $G$, a contradiction.
Thus, $H$ is an oriented graph.
By a simple averaging argument, there exists a vertex $y \in V(H)$ such that
\begin{align}
	d^-(y) \ge d^+(y) \ge d(y) + \delta^c(G)- n. \label{eqn:k3}
\end{align}
Recall that, for all $z \in N^-(y)$, $c(zy) = c(xy)$ and so $d^-(y) \le d(y) - \delta^c(G)$.
Therefore, by~\eqref{eqn:k3}, we have $\delta^c(G) \le n/2$ and so $\delta^c(G) = n/2$ by our assumption, so that $n$ is even. For all $y \in V(H)$, equality must hold for~\eqref{eqn:star},~\eqref{eqn:d+(y)} and~\eqref{eqn:k3} .

If $\Delta(G) = n/2$, then $G$ is $n/2$-regular.
By~\eqref{eqn:d+(y)}, we deduce that $G[N(x)]$ is empty and so $G$ is a properly coloured~$K_{n/2,n/2}$, a contradiction.

Thus, we may assume that $\Delta(G) > n/2$. 
Since equality holds for~\eqref{eqn:star}, (by relabeling the colours if necessary) we may assume that $N(x) = \{y_1,\dots y_{n/2-1}\} \cup Y'$ such that $c(xy_i) = i$ for $i \in [n/2-1]$ and $c(xy') =0$ for all $ y' \in Y'$ and $|Y'| \ge 2$. 
Consider $y' \in Y'$. 
Recall that each monochromatic subgraph of $G$ is a star and if $z \in N^-(y')$, then $c(zy') = c(xy')$.
Thus we have $d^-(y') = 0$ and so since equality holds for~\eqref{eqn:k3}, we have
\begin{align}
d(y') = n/2 \text{ for all $y' \in Y'$}. \label{eqn:y'}
\end{align}
Again by~\eqref{eqn:star} for $x y_1$, we deduce that $d_{1}(y_1)  = \Delta(G) - \delta^c(G) +1$ and so $	d(y_1) = \Delta(G)$.

We now repeat the whole proof with $y_1$ playing the role of~$x$. 
We deduce that $x$ is in the corresponding~$Y'$ for~$y_1$. 
By~\eqref{eqn:y'}, we have $d(x) = n/2 < \Delta(G) = d(x)$, a contradiction. 
\end{proof}

Given a digraph~$G$, we can consider an edge-coloured graph~$H$ of the base graph of~$G$ such that for each $xy \in E(G - G^{\pm})$, $c(xy) = y$ and the subgraph $G^{\pm}$ of $H$ is rainbow coloured with a new set of colours.
Note that $\delta^c(H) = \delta^+(G)$. 
We get the following corollary from Theorem~\ref{thm:rainbowK_3}.

\begin{corollary} \label{cor:directedK_3}
Let $G$ be a digraph on $n$ vertices with $\delta^+(G) > n/2$. 
Then there exists a $K_3$ with vertex set~$T$ in the base graph of~$G$ such that $\delta^+(G[T]) \ge 1$.
\end{corollary}

\section{Absorption lemma} \label{sec:absorption}

In this section, we consider an absorption lemma. We need the following notation from Markstr\"om and the first author~\cite{LoMarkstrom2015}.

Let $r \in \mathbb{N}$ and $G$ be an edge-coloured graph on $n$ vertices. 
Given $s \in \mathbb{N}$ and vertices $x, y\in V(G)$, we say that the vertex set $S\subseteq V(G)$ is an \emph{$(x,y;K_r)$-connector of length~$s$} if $S\cap \{x, y\}=\emptyset$, $|S|=rs-1$ and both $G[S\cup \{x\}]$ and $G[S\cup \{y\}]$ contain perfect rainbow-$K_r$-tilings.
Given an integer $s\ge 1$ and a constant $\eta>0$, two vertices $x,y \in V(G)$ are \emph{$(s,\eta)$-close} to each other if there exist at least $\eta n^{rs-1}$ $(x, y;K_r)$-connectors of length~$s$ in~$G$.
A subset $U\subseteq V(G)$ is said to be \emph{$(s,\eta;K_r)$-closed in~$G$} if every two vertices in $U$ are $(s, \eta;K_r)$-close to each other.
If $V(G)$ is $(s,\eta;K_r)$-closed in~$G$ then we simply say that \emph{$G$ is $(s,\eta;K_r)$-closed}.

\begin{lemma}[{Lo and Markstr\"om~\cite[Lemma~1.1]{LoMarkstrom2015}}] \label{lemma:LoMarkstrom}
Let $1/n \ll \phi \ll \eta, 1/s,1/r$.
Let $G$ be an edge-coloured graph on $n$ vertices that is $(s,\eta;K_r)$-closed.
Then there exists an absorbing set $W\subseteq V(G)$ of order at most $\eta n$ so that for every $U \subseteq V(G) \setminus W$ with $|U| \le \phi n$ and $|U|\in r\mathbb{N}$, both $G[W]$ and $G[U\cup W]$ contain perfect rainbow-$K_r$-tilings.
\end{lemma}

Although, the original statement is for non-edge-coloured graphs, the argument can be easily adapted to the edge-coloured setting.

\begin{proposition} \label{prop:1-closed}
Let $1/n \ll \eta \ll \eps,1/r$.
Let $G$ be an edge-coloured graph on $n$ vertices with $\delta^c(G)\ge (1-\frac{1}{2(r-1)} + \eps) n$.
Then $G$ is $(1, \eta;K_r)$-closed.
\end{proposition}

\begin{proof}
Let $\eta \ll \eta' \ll \eps,1/r$. 
Let $x,y \in V(G)$ be distinct. 
Note that 
\begin{align*}
	d^c(x)+d^c(y) - n \ge (r-2) n/(r-1). 
\end{align*}
Let $Z$ be a subset of~$N(x) \cap N(y)$ of size $(r-2) n/(r-1)$ such that, for distinct $z,z' \in Z$, we have $c(xz) \ne c(xz')$ and $c(yz) \ne c(yz')$.
Define an auxiliary digraph~$H$ on~$Z$ such that we have $zz' \in E(H)$ if and only if $zz' \in E(G[Z])$, $c(xz) \ne c(zz')$ and $c(yz) \ne c(zz')$.
Note that $\delta^+(H) \ge  (1-\frac{3}{2(r-1)} + \eps/2) n$ and so 
\begin{align*}
	e(H^{\pm}) \ge \left(1-\frac{3}{2(r-1)} + \eps/2 \right) n |Z| - |Z|^2/2 = 
	\frac{1}{2}\left(\frac{r-3}{r-2} + \frac{(r-1)\eps}{r-2} \right)|Z|^2.
\end{align*}
Let $t = r^3$. 
By (the supersaturate version of) the Erd\H{o}s--Stone--Simonovitz theorem, $H^{\pm}$ contains $\eta' n^{(r-1)t}$ many $K_{r-1}(t)$. 
Consider one $K_{r-1}(t)$ in~$H^{\pm}$ with vertex classes $V_1,\dots, V_{r-1}$. 
Let $K_x$ be the complete $r$-partite graph in $G$ obtained from $K_{r-1}(t)$ by joining $x$ to all vertices, and define~$K_y$ similarly. 
Note that both $K_x$ and~$K_y$ are properly edge-coloured in~$G$. 
We can now greedily pick $v_i \in V_{i}$ for $i \in [r-1]$ in turn such that $xv_1 \dots v_i$ and $yv_1 \dots v_i$ are both rainbow $K_{i+1}$ in~$G$. 
Indeed, since given $v_1 \dots v_i$ with $i \le r-2$, each of $xv_1 \dots v_i$ and $yv_1 \dots v_i$ in~$G$ contain at most $\binom{i}2+(i-1) = (i+2)(i-1)/2 \le r(r-3)/2$ colours. 
Since $t > (i+2)r(r-3)/2$, there exists a vertex $v_{i+1} \in V_{i+1}$ such that $v_{i+1}$ is joined to $\{x,y, v_1, \dots, v_i\}$ with new colours. 
Hence both $xv_1\dots v_{r-1}$ and $yv_1\dots v_{r-1}$ form rainbow~$K_r$ in~$G$. 
Since $\eta \ll \eta'$, we deduce that there are at least $\eta n^{r-1}$ such choices for $v_1\dots v_{r-1}$ and so $G$ is $(1, \eta; K_r)$-closed. 
\end{proof}


\section{Edge-coloured graphs to digraphs} \label{sec:convert}

The following proposition allows us to convert edge-coloured graphs to digraphs, which was introduced by Czygrinow, Molla and Nagle~\cite{CzygrinowMollaNagle}.

\begin{proposition} \label{prop:conversion}
Let $G$ be an edge-coloured graph on $n$ vertices with $\delta^c(G)  \ge \delta n$. 
Then there exists a digraph~$H$ on~$V(G)$ such that $\delta^+(H) \ge \delta^c(G) - \sqrt{n}$, the base graph of $H$ is a subgraph of~$G$, for all $u \in V(H)$, $G[u,N^+_H(u)]$ is rainbow and, for all $uv \in E(H)$, there are at most $\sqrt{n}$ vertices $w \in N^-_H(u)$ with $c(uv) = c(uw)$. 
Moreover, given $v \in V(G)$, the number of pairs $(x,y) \in V(G) \times V(G)$ such that  $xv,vy \in E(H)$ and $c(xv) = c(vy)$ is at most~$n$. 
\end{proposition}

\begin{proof}
Let $H$ be an edge-coloured critical subgraph of~$G$ with $\delta^c(H) = \delta^c(G)$. 
Each colour induces vertex-disjoint stars in~$H$ and we orient all edges toward the centre of each star. 
If a star has size at most $\sqrt{n}$, then we arbitrarily pick one of its leaves and add an arc from its centre to this leaf. 
Note that each vertex can be a centre of at most $\sqrt{n}$ stars that each have at least $\sqrt{n}$ leaves. 
Thus $\delta^+(H) \ge \delta^c(G) - \sqrt{n}$. 
By our construction, we have for all $u \in V(H)$, $G[u,N^+_H(u)]$ is rainbow and 
 for all $uv \in E(H)$, there are at most $\sqrt{n}$ vertices $w \in N^-_H(u)$ with $c(uv) = c(uw)$. 

To see the moreover statement, observe that given a directed path~$xvy$ in~$H$ with $c(xv) = c(vy)$, $y$ is uniquely determined by~$x$. 
\end{proof}

Note that a rainbow~$K_r$ in~$G$ translates to a digraph~$K$ on $r$ vertices with $\delta^+(K) \ge r-2$ whose base graph is complete.
However the converse false. 
We now show a weaker result that most of the blow-ups of~$K$ are properly coloured in~$G$. 
Then one can find a rainbow~$K_r$ using a result from Axenovich, Jiang, Tuza~\cite{AxenovichJiangTuza2003} and independently by Keevash, Mubayi, Sudakov and Verstra\"ete~\cite{KeevashMubayiSudakovVerstraete2007}.

\begin{corollary} \label{cor:conversionK_t}
Let $t,r \in \mathbb{N}$ and $r\ge 3$. Let $G$ and $H$ be as defined in Proposition~\ref{prop:conversion}.
Let $K$ be a digraph on $r$ vertices such that the base graph of $K$ is complete and $\delta^+(K) \ge r-2$. 
Then there are at most $rtn^{rt-1}$ many copies of $K(t)$ in~$H$ such that $K(t)$ is not properly coloured in~$G$. 
\end{corollary}

\begin{proof}
Consider a copy of $K(t)$ in~$H$ with vertex classes $V_1, \dots, V_r$. 
Suppose that $K(t)$ is not properly coloured in~$G$, say $c(vx) = c(vy)$. 
Since $\delta^+(K) \ge r-2$, we may assume that $xvy$ is a directed path in~$H$. 
By the moreover statement of Proposition~\ref{prop:conversion}, there are at most $n$ choices of $(x,y)$ for a given~$v$. 
Therefore, there are at most $rt n^{rt-1}$ many such $K(t)$.
\end{proof}

\begin{lemma}[{\cite{AxenovichJiangTuza2003,KeevashMubayiSudakovVerstraete2007}}] \label{lma:KMSV}
Let $r \in \mathbb{N}$ and $r\ge 3$. Every proper edge-colouring of $K_r(r^3)$ contains a rainbow $K_r$.
\end{lemma}

\subsection{Regularity lemma}\label{sec:regularity}

Let $G$ be a bipartite graph with vertex classes~$A$ and~$B$. 
For non-empty sets $X\subseteq A$, ${Y\subseteq B}$, we define the \emph{density of $G[X,Y]$} to be $d_G(X, Y)=e_G(X,Y)/|X||Y|$.
Let $\eps>0$. 
We say that $G$ is \emph{$\eps$-regular} if, for all sets $X \subseteq A$ and $Y \subseteq B$, with $|X|\geq \eps |A|$ and $|Y| \geq \eps |B|$ we have $|d_G(A,B) - d_G(X,Y)| < \eps$.
The following simple result follows immediately from this definition.


\begin{proposition}\label{prop:slice}
Let $\eps>0$ and $G$ be a bipartite graph with vertex classes $A$ and $B$.
Suppose that $G$ is $\eps$-regular with density~$d$.
Let $c > \eps$,  $A' \subseteq A$ and $B' \subseteq B$ with  $|A'| \ge c |A|$ and $|B'| \ge c |B|$.
Then $G[A'\cup B']$ is $2\eps/c$-regular with density at least $d - \eps$.
\end{proposition}

Let $G$ be a digraph. 
We say that $\mathcal{Q}$ is an \emph{$(\eps,d,m,k)$-regular partition of~$G$}, if 
\begin{enumerate}[label={(Q\arabic*)}]
	\item $\mathcal{Q}= \{V_0, V_1,  \dots, V_k\}$ is a partition of~$V(G)$; 
	\item \label{itm:Q1} $|V_0| \le \eps |V(G)|$;
	\item \label{itm:Q2} $|V_1| = \dots = |V_k| = m$;
	\item \label{itm:Q3} for all distinct $i,j \in [k]$, the graph~$G[V_i, V_j]$ is $\eps$-regular and has density either $0$ or at least~$d$;
	\item \label{itm:Q4} for all distinct $i,j \in [k]$, the graph~$G^{\pm}[V_i, V_j]$ is $\eps$-regular and has density either $0$ or at least~$d$;
	\item \label{itm:Q5} for all $i \in [k]$, $G[V_i]$ is empty.
\end{enumerate}

We use the regularity lemma for digraphs attributed to Alon and Shapira~\cite{AlonShapira}. 
In particular, we will use its degree form that can be derived from the standard version (see~\cite[{Lemma~7.3}]{KOsurvey} for a sketch of the proof for the undirected case that can be adapted to the digraph setting). 

\begin{lemma}[Degree form of the diregularity lemma~\cite{AlonShapira}] \label{lma:szemeredi}
Let $1/n \ll 1/N \ll \eps,\delta^+,d$. 
Let $G$ be a digraph on $n$ vertices with $\delta^+(G) \ge \delta^+ n $.
Then there exists a spanning subgraph~$G'$ of~$G$ and an $(\eps, d,m,k)$-regular partition $\mathcal{Q}= \{V_0, V_1,  \dots, V_k\}$ of~$G'$ with $1/\eps \leq k \leq N$ and $\Delta ( G - G') \le  (d + \eps)n$.
\end{lemma}

Let $m,k \in \mathbb{N}$ and $\eps, d>0$.
Let $G$ be a digraph on $n$ vertices and $\mathcal{Q} = \{V_0, V_1, \dots, V_k\}$ an $(\eps, d,m,k)$-regular partition of~$G$. 
We say that a digraph $R$ is an \emph{$(\eps, d)$-reduced digraph respecting~$(G,\mathcal{Q})$} if the followings hold. 
The vertex set of~$R$ is a subset of clusters $\{V_i : i \in [k]\}$.
For each $U,U' \in V(R)$, if $U U'$ is a directed edge of~$R$, then the subgraph $G[U,U']$ is $\eps$-regular and has density at least~$d$.
Moreover, if $UU'$ is in $R^{\pm}$, then the subgraph $G^{\pm}[U,U']$ is $\eps$-regular and has density at least~$d$.

Unlike its undirected version, the reduced digraph is not uniquely defined. 
For example, if $G$ is an oriented graph such that $G[U,U']$ and $G[U',U]$ are $\eps$-regular and have density at least~$d$, then a reduced digraph can contain either $UU'$ or $U'U$ but not both. 

\begin{lemma} \label{lma:reducedgraph}
Let $1/n \ll 1/k \ll \eps,\delta^+,d$. 
Let $G$ be a digraph on $n$ vertices with $\delta^+(G) \ge \delta^+n$ and an $(\eps, d,m,k)$-regular partition $\mathcal{Q}= \{V_0, V_1,  \dots, V_k\}$.
Then there exists an $(\eps, d)$-reduced digraph~$R$ respecting $(G,\mathcal{Q})$ such that $\delta^+(R) \ge (\delta^+ -2 \eps)k$.
\end{lemma}

The proof is very similar to the proof of Lemma~3.2 in~\cite{KellyKuhnOsthus}.

\begin{proof}
Recall that if $d_{G^{\pm}}(V_i,V_j) >0$, then $d_{G^{\pm}}(V_i,V_j)  \ge d$.
We define a random digraph~$R$ on~$\{V_i : i \in [k]\}$ as follows. 
For all $i, j \in [k]$ with $i < j$, (independent of other choices)
\begin{itemize}
	\item if $d_{G^{\pm}}(V_i,V_j) >0$, then $V_iV_j \in E(R^{\pm})$ with probability~$1$;
	\item if $d_{G^{\pm}}(V_i,V_j) =0$ and $d_{G}(V_i,V_i) >0$, then $V_iV_j \in E(R)$ with probability $\frac{ d_{G}(V_i,V_j) }{d_{G}(V_i,V_j)+d_{G}(V_j,V_i) }$, otherwise $V_jV_i \in E(R)$. 
\end{itemize}
Note that 
\begin{align*}
		\mathbb{P}( V_iV_j \in E(R) ) & = \mathbb{P}( V_iV_j \in E(R-R^{\pm}) )+ \mathbb{P}( V_iV_j \in E( R^{\pm} ) )\\
		& = \mathbbm{1}(d_{G^{\pm}}(V_i,V_j) =0) \frac{ d_{G}(V_i,V_j) }{d_{G}(V_i,V_j)+d_{G}(V_j,V_i) } + \mathbbm{1}(d_{G^{\pm}}(V_i,V_j) >0)\\
		&\ge d_{G}(V_i,V_j).
\end{align*}
Hence
\begin{align*}
	\mathbb{E} ( d^+_R(V_i)) 
	\ge \sum_{j \in [k] \setminus i} d_{G}(V_i,V_j) 
	= \frac{ \sum_{v \in V_i} d^+_{G \setminus V_0}(v)}{m^2} 
	\ge \frac{ (\delta^+ - \eps)n}{m} \ge (\delta^+ - \eps)k.
\end{align*}
Therefore, a standard application of the Chernoff and union bounds shows that with positive probability, $\delta^+(R) \ge (\delta^+ - 2\eps)k$.
\end{proof}

We say that $\varphi$ is a \emph{homomorphism  from a digraph~$H$ to a digraph~$G$} if $\varphi : V(H) \rightarrow V(G)$ is such that if $uv \in E(H)$, then $\varphi(u)\varphi(v) \in E(G)$.

We also use the `counting lemma'. 

\begin{lemma}[Counting lemma] \label{lma:counting}
Let $1/m \ll c \ll \eps \ll d,1/k,1/\Delta,1/s$.
Let $G$ be a digraph with an $(\eps, d, m ,k)$-regular partition~$\mathcal{Q}$.
Let $R$ be an $(\eps,d)$-reduced digraph respecting~$(G, \mathcal{Q})$ with $V(R) = \{V_i: i \in [k]\}$.
Let $H$ be a digraph on~$s$ vertices with $\Delta(H) \le \Delta$ and a homomorphism~$\varphi$ from~$H$ to~$R$. 
Let $u,u' \in V(H)$ with $uu',u'u \notin E(H)$. 
Then there exist $V'_{\varphi(u)} \subseteq V_{\varphi(u)}$ and $V'_{\varphi(u')} \subseteq V_{\varphi(u')}$ with $|V'_{\varphi(u)} |, |V'_{\varphi(u')} |\ge (1-3\Delta \eps)m $ such that, for every pair $(x,x') \in V'_{\varphi(u)} \times V'_{\varphi(u')}$, there are at least $cn^{s-2}$ injective homomorphisms~$\varphi'$~from~$H$ to~$G$ such that $\varphi'(v) \in \varphi(v)$ for all~$v \in V(H)$, $\varphi'(u) = x$ and $\varphi'(u') = x'$.
\end{lemma}


\section{Almost rainbow-$K_r$-tilings} \label{sec:almosttiling}

In this section we focus on finding a rainbow-$K_r$-tiling covering almost all vertices of our edge-coloured graph. Recall that one of the key components of our proof is to translate the edge-colouring problem into a digraph problem, so we use the tools considered in~\Cref{sec:convert} to do this. First we must introduce the notion of fractional tilings.

Let $r \in \mathbb{N}$ with $r \geq 3$. 
Let $\mathcal{K}_{r,r-2}$ be the set of all digraphs on $r$ vertices with minimum out-degree at least $r-2$ and complete base graph. 
We write an \emph{$(r,r-2)$-tiling} for a $\mathcal{K}_{r,r-2}$-tiling.\footnote{
For a family $\mathcal{H}$ of digraphs, an \emph{$\mathcal{H}$-tiling} is a collection of vertex-disjoint copies of members of~$\mathcal{H}$.}
Let $G$ be a digraph on $n$~vertices.
Let $\mathcal{K}_{r,r-2}(G)$ be the set of all copies of $\mathcal{K}_{r,r-2}$ in~$G$. 
A \emph{fractional $(r,r-2)$-tiling} is a function $\omega^* : \mathcal{K}_{r,r-2}(G) \rightarrow [0,1]$ such that 
\begin{align*}
\sum_{v \in K \in \mathcal{K}_{r,r-2}(G)} \omega^* (K) & \le 1 \textrm{ for all $ v \in V(G)$.}
\end{align*}
We say that $\omega^*$ is \emph{perfect} if equality holds for all $v \in V(G)$. 

We now show that finding an almost rainbow-$K_r$-tiling in an edge-coloured graphs can be reduced to finding a perfect fractional $(r,r-2)$-tiling in digraphs.

\begin{lemma}\label{lem:almost_tiling}
Let $1/n \ll 1/t \ll \eps \ll \gamma, 1/r$ with $r\geq 3$.
Suppose that every digraph~$R$ on~$t$ vertices with $\delta^+(R) \ge \delta t$ contains a perfect fractional $(r,r-2)$-tiling.
Let $G$ be an edge-coloured graph on $n$ vertices with $\delta^c(G)  \ge (\delta +3\gamma)n$. 
Then $G$ contains a rainbow-$K_r$-tiling covering all but at most $3\eps^{1/2} n$ vertices. 
\end{lemma}

\begin{proof}
By Proposition~\ref{prop:conversion} and Corollary~\ref{cor:conversionK_t}, there exists a digraph~$H$ on~$V(G)$ such that $\delta^+(H) \ge (\delta+2\gamma)n$ and, for each $K \in \mathcal{K}_{r,r-2}$, there are at most $r^4n^{r^4-2}$ many copies of $K(r^3)$ in~$H$ such that $K(r^3)$ is not properly coloured in~$G$. 

Let $ 1/n \ll 1/N \ll 1/t \ll \eps \ll d \ll \gamma$. 
Apply Lemma~\ref{lma:szemeredi} (with $H$ playing the role of~$G$) and obtain a spanning subgraph~$H'$ of~$H$ and an $(\eps, d,m,t)$-regular partition $\mathcal{Q}= \{V_0, V_1,  \dots, V_t\}$ of~$H'$ with $\eps^{-1} \leq t \leq N$ and $\delta^+(H') \ge (\delta+ \gamma)n$.
Apply Lemma~\ref{lma:reducedgraph} (with $H',\eps,\delta + \gamma$ playing the roles of~$G,\eps,\delta^+$) and obtain an $(\eps, d)$-reduced digraph~$R$ respecting~$(H',\mathcal{Q})$ satisfying $\delta^+(R) > \delta t$.
By our assumption, $R$ contains a perfect fractional $(r,r-2)$-tiling~$\omega$.

We now construct a rainbow-$K_r$-tiling~$\mathcal{T}$ in~$G$ covering all but $2 \eps^{1/2} m$ vertices of each~$V_i$ as follows (which implies the lemma). 
Initially, we set $\mathcal{T}$ to be empty. 
Let $ 1/m \ll c \ll \eps$. 
Pick a $K \in \mathcal{K}_{r,r-2}(R)$ that has not yet been considered. 
Without loss of generality, $V(K) = \{ V_i: i \in [r]\}$. 
For each $ i \in [r]$, we pick $U_i \subseteq V_i \setminus V(\mathcal{T})$ of size $((1- 2\eps^{1/2} \omega(K) ) + 2 \eps^{1/2}) m$ (which will be possible by our construction). 
By Proposition~\ref{prop:slice} and Lemma~\ref{lma:counting}, the number of~$K(r^3)$ in~$H'[\bigcup_{i \in [r]}U_i]$ is at least~$cn^{r^4}$.
At least one of them will be properly coloured in~$G$. 
By Lemma~\ref{lma:KMSV}, there is a rainbow~$K_r$ in $G$ with one vertex in each $U_i$. 
Repeat this argument until we have find $(1- 2\eps^{1/2}) \omega(K) m$ vertex-disjoint rainbow~$K_r$ in $G[\bigcup_{i \in [r]}U_i]$.
We repeat this argument for all $K \in \mathcal{K}_{r,r-2}(R)$.
\end{proof}

\subsection{Fractional $(r,r-2)$-tilings and Farkas' lemma.}

We require Farkas' lemma to find a fractional $(r,r-2)$-tiling. 
\begin{lemma}[{Farkas' lemma, see \cite[P.257]{MR0859549}}] \label{lma:Farkas}
A system of equations $y A = b $, $y \ge 0$ is solvable if and only if the system $A x \ge 0 $, $bx <0$ is unsolvable.
\end{lemma}

A $k$-uniform hypergraph~$\mathcal{F}$ or $k$-graph for short is a pair $(V(\mathcal{F}),E(\mathcal{F}))$ such that $E(\mathcal{F}) \subseteq \binom{V(\mathcal{F})}{k}$.
A \emph{fractional matching} is a function $\omega^* : E(\mathcal{F}) \rightarrow [0,1]$ such that 
\begin{align*}
\sum_{v \in e \in E(\mathcal{F})} \omega^* (e) & \le 1 \textrm{ for all $ v \in V(\mathcal{F})$.}
\end{align*}
We say that $\omega^*$ is \emph{perfect} if equality holds for all $v \in V(\mathcal{F})$.

\begin{corollary} \label{cor:fractionaltiling}
Let $\mathcal{F}$ be a $k$-graph on $n$ vertices that does not contains a perfect fractional matching. 
Then there exists a vertex weighting $\omega : V(\mathcal{F}) \rightarrow [0,1]$ such that 
\begin{align}
	\sum_{v \in e} \omega(v) &\ge 1 \textrm{ for all $ e \in E(\mathcal{F})$,} &
	\min_{v \in V(\mathcal{F})} \omega(v) &= 0, &
	\sum_{v \in V(\mathcal{F})} \omega(v) &< n/k. \nonumber
\end{align}
\end{corollary}

\begin{proof}
Let $A$ be the adjacency matrix of~$\mathcal{F}$ with rows and columns representing~$E(\mathcal{F})$ and~$V(\mathcal{F})$ respectively, such that $A_{e,v} =1$ if and only if $v \in e$ for all~$e \in E(\mathcal{F})$ and all~$v \in V(\mathcal{F})$.
Since $\mathcal{F}$ has no perfect fractional matching, taking $\textbf{y} = ( \omega_0(e): e \in E(\mathcal{F}))$ and $\textbf{b}=(1,\dots,1)$, Farkas' lemma (Lemma~\ref{lma:Farkas}) implies that there exists a weighting function $\omega: V(\mathcal{F}) \rightarrow \mathbb{R}$ such that 
\begin{align}
	  \sum_{v \in e} \omega(v) \ge 0 \text{ for all $ e \in E(\mathcal{F})$ and } \sum_{v \in V(\mathcal{F})} \omega(v) <0. \label{eqn:Farkas}
\end{align}
Let $V(\mathcal{F}) = \{v_i : i \in [n] \}$ such that $(\omega (v_i))_{i \in [n]}$ is a decreasing sequence. 
Since $\sum_{v \in V(\mathcal{F})} \omega(v)<0$, then $\omega(v_n)<0$.
Thus, by multiplying through by $|\omega(v_n)|^{-1}$, we may assume that~$\omega(v_{n}) =-1$.
We further assume that $\omega(v) \le k-1$ for all $v \in V(\mathcal{F})$, as \eqref{eqn:Farkas} still holds after we replace $\omega(v)$ with $\min \{ \omega(v), k-1 \}$.
Finally, we apply the linear transformation $(\omega(v)+1)/k$ for $v \in V(\mathcal{F})$, which scales~$\omega$ so that its image now lies in the interval~$[0,1]$ and $\omega$ satisfies the following inequalities
\begin{align}
	\sum_{v \in e} \omega(v) \ge 1 \textrm{ for all $ e \in E(\mathcal{F})$ and } \sum_{v \in V(\mathcal{F})} \omega(v) < n/k \nonumber
\end{align}
as required. 
\end{proof}

\subsection{Perfect fractional $(r,r-2)$-tilings for $r \ge 4$.}


Suppose that $G$ is digraph without a perfect fractional $(r,r-2)$-tiling.
Let $\mathcal{F}= \mathcal{F}(G)$ be the $r$-graph on $V(G)$ such that every edge in~$\mathcal{F}$ is a member of $\mathcal{K}_{r,r-2}(G)$. There is no perfect fractional matching in $\mathcal{F}$, since this corresponds to a perfect fractional $(r,r-2)$-tiling in $G$. By Corollary~\ref{cor:fractionaltiling}, there exists a vertex weighting~$\omega$. 
Let $v_0 \in V(G)$ be such that $\omega(v_0) = 0$. 
To obtain a contradiction, a typical approach is to show that 
$G^{\pm}[N^+(v_0)]$ contains a $K_{r-1}$-tiling of size $n/r$, since each copy in this tiling together with $v_0$ form an element of $\mathcal{K}_{r,r-2}(G)$.

\begin{lemma} \label{lma:fractionalrainbowK_4r}
Let $1/n \ll \gamma,1/r \le 1/5$. 
Let $G$ be a digraph on $n$ vertices with $\delta^+(G) \ge ( 1 - \frac{1}{(2r-3)r}+ \gamma)n$.
Then $G$ contains a perfect fractional $(r,r-2)$-tiling.
\end{lemma}

\begin{proof}
Let $\delta = 1 - ((2r-3)r)^{-1}+ \gamma$.
Suppose to the contrary.
Let $\mathcal{F}$ be the $r$-graph on $V(G)$ such every $S \subseteq \binom{V(G)}{r}$ is an edge in~$\mathcal{F}$ if and only if $S \in \mathcal{K}_{r,r-2}(G)$.
By Corollary~\ref{cor:fractionaltiling}, there exists a vertex weighting $\omega : V(\mathcal{F}) \rightarrow [0,1]$ such that 
\begin{align}
	\sum_{v \in  K} \omega(v) &\ge 1 \textrm{ for all $ K \in \mathcal{K}_{r,r-2}(G)$,} &
	\min_{v \in V(G)} \omega(v) &= 0, &
	\sum_{v \in V(G)} \omega(v) &< n/r. \nonumber
\end{align}
Let $v_0 \in V(G)$ be such that $\omega(v_0) = 0$.
Let $\mathcal{T}$ be a $K_{r-1}$-tiling in $G^{\pm}[ N^+(v_0) ]$ with $| \mathcal{T} |$ maximal. 

Suppose that $|\mathcal{T}| \ge n/r$.
Each~$K \in\mathcal{T}$ together with $v_0$ is a member of~$\mathcal{K}_{r,r-2}(G)$ and so $\sum_{v \in K} \omega(v) \ge 1$ for all $K \in\mathcal{T}$.
Hence we have 
\begin{align*}
n/r & > \sum_{v \in V(G)} \omega(v) \ge  \sum_{K \in \mathcal{T}} \sum_{v \in K} \omega(v)  \ge |\mathcal{T}| \ge n/r,
\end{align*}
a contradiction. 

Hence we may assume that $|\mathcal{T}| < n/r$. 
Let $U = N^+(v_0) \setminus V( \bigcup\mathcal{T}) $, so
\begin{align*}
	|U| \ge \delta n - \frac{(r-1)n}{r} \ge \frac{2(r-2)n}{r(2r-3)}.
\end{align*}
Note that $G^{\pm}[U]$ is $K_{r-1}$-free. 
However, 
\begin{align*}
	e(G^{\pm}[U]) \ge  \binom{|U|}{2} - (1- \delta) n |U| 
	\ge \binom{|U|}{2} - \frac{(1- \gamma) n |U|}{r(2r-3)} 
	> \left( 1 - \frac{1}{r-2}  \right)\binom{|U|}2
\end{align*}
contradicting Tur\'an's theorem.
\end{proof}

For the case $r = 4$, we use the following density version of the Corr\'adi--Hajnal theorem of Allen, B\"ottcher, Hladk\'y and Piguet~\cite{ABHP2015} to find a large $K_3$-tiling in $G^\pm[N^+(v_0)]$. 
We will only state a simplified version of their result.

\begin{theorem}[Allen, B\"ottcher, Hladk\'y and Piguet~\cite{ABHP2015}] \label{thm:ABHP}
Let $1/n \ll \gamma, \alpha$. 
Let $G$ be a graph on $n$ vertices with 
\begin{align*}
e(G) \ge \max \left\{ \frac{1}4 (1+2\alpha-\alpha^2) , \left( \frac14+2\alpha^2\right), 2 \alpha (1-\alpha), \frac12-3\alpha +9\alpha^2 \right\} n^2+ \gamma n^2.
\end{align*}
Then $G$ contains $\alpha n$ vertex-disjoint~$K_3$.  
\end{theorem}

\begin{lemma} \label{lma:fractionalrainbowK_4}
Let $1/n \ll \gamma$.
Let $G$ be a digraph on $n$ vertices with $\delta^+(G) \ge (2+\sqrt{13/2}+ \gamma)n/5$.
Then $G$ contains a perfect fractional $(4,2)$-tiling.
\end{lemma}

\begin{proof}
Let $\delta = (2+\sqrt{13/2}+ \gamma)/5$.
Suppose to the contrary.
Let $\mathcal{F}$ be $4$-graph on $V(G)$ such that every $S \subseteq \binom{V(G)}{4}$ is an edge in~$\mathcal{F}$ if and only if $S \in \mathcal{K}_{4,2}(G)$.
By Corollary~\ref{cor:fractionaltiling}, there exists a vertex weighting $\omega : V(\mathcal{F}) \rightarrow [0,1]$ such that 
\begin{align}
	\sum_{v \in  K} \omega(v) &\ge 1 \textrm{ for all $ K \in \mathcal{K}_{4,2}(G)$,} &
	\min_{v \in V(G)} \omega(v) &= 0, &
	\sum_{v \in V(G)} \omega(v) &< n/4. \nonumber
\end{align}
Let $v_0 \in V(G)$ be such that $\omega(v_0) = 0$.

Let $U \subseteq N^+(v_0)$ of size $\delta n$.
Note that $\delta^+(G[U]) \ge (2\delta -1)n$ and so
\begin{align*}
	e(G^{\pm}[U]) \ge (2\delta -1) \delta n^2 - \binom{\delta n }{2} = \frac{1}{2}\delta (3 \delta-2) n^2. 
\end{align*}
By our choice of $\delta$ and Theorem~\ref{thm:ABHP}, $G^{\pm}[U]$ contains a $K_3$-tiling~$\mathcal{T}$ of size $n/4$.
Each~$K \in\mathcal{T}$ together with $v_0$ is a member of~$\mathcal{K}_{4,2}(G)$ and so $\sum_{v \in K} \omega(v) \ge 1$ for all $K \in\mathcal{T}$.
Hence we have 
\begin{align*}
n/4 & > \sum_{v \in V(G)} \omega(v) \ge  \sum_{K \in \mathcal{T}} \sum_{v \in K} \omega(v) \ge n/4,
\end{align*}
a contradiction. 
\end{proof}

\subsection{Perfect fractional $(3,1)$-tilings}

When $r =3$, we prove the following lemma. 

\begin{lemma} \label{lma:fractionalrainbowK_3}
Let $1/n \ll  1$.
Let $G$ be a digraph on $n$ vertices with $\delta^+(G) \ge 5n/6$.
Then $G$ contains a perfect fractional $(3,1)$-tiling.
\end{lemma}

We will use a classical result of Erd\H{o}s and Gallai on the Tur\'an number for matchings instead of Theorem~\ref{thm:ABHP}.

\begin{theorem}[Erd\H{o}s and Gallai~\cite{ErdosGallai1961}] \label{thm:ErdosGallai}
Let $G$ be a graph on $n$ vertices with 
\begin{align*}
e(G) > \max \left\{ \binom{2 t -1}{2}, \binom{n}{2}-\binom{n-t+1}{2} \right\}.
\end{align*}
Then $G$ contains a matching of size $t$. 
\end{theorem}

The proof of Lemma~\ref{lma:fractionalrainbowK_3} is similar to the proof of Lemma~\ref{lma:fractionalrainbowK_4}.
In particular, to obtain a contradiction, we seek a matching in $G^{\pm}[N^+(v_0)]$ of size~$n/3$.
However, we may not have sufficiently many edges in $G^{\pm}[N^+(v_0)]$ to apply the Erd\H{o}s--Gallai theorem (Theorem~\ref{thm:ErdosGallai}). 
In this case, we deduce that $e(G[N^+(v_0), V(G) \setminus N^+(v_0)])$ must be large. 
This means that some vertex $v \in V(G) \setminus  N^+(v_0)$ has large indegree. 
We then bound~$\omega(v)$ from below in terms of $d^-(v)$, see Claim~\ref{clm:weightbydegree}.

Let $\mathcal{F}$ be a $3$-graph.
For $v \in V(\mathcal{F})$, let $\mathcal{F}_v$ be \emph{the link graph of~$v$}, that is, the graph on $V(\mathcal{F})$ such that $S$ is an edge in~$\mathcal{F}_v$ if and only if $ S \cup \{v\}$ is an edge in~$\mathcal{F}$. In particular, if $\mathcal{F}$ is the $3$-graph on~$V(G)$ such that every edge in $\mathcal{F}$ corresponds to a member of~$\mathcal{K}_{3,1}$, then a matching in~$G^{\pm}[N^+(v)]$ is a matching in~$\mathcal{F}_{v}$, for every $v \in V(G)$.

Let $J_3$ be the digraph on vertices $x,y,z$ with edge set $\{ xy, xz , yz, zy\}$. 
Note that every member of $\mathcal{K}_{3,1}$ contains a cyclic triangle $C_3$ or  $J_3$ as a subgraph.

\begin{proof}[Proof of Lemma~\ref{lma:fractionalrainbowK_3}]
Let $\delta = 5/6$. 
By deleting edges in~$G$ if necessary, we may assume that $d^+(v) = \delta n $ for all $v \in V(G)$. 
Suppose to the contrary that $G$ does not contain a perfect fractional $(3,1)$-tiling.
Let $\mathcal{F}$ be the $3$-graph on $V(G)$ such that every $S \in \binom{V(G)}{3}$ is an edge in~$\mathcal{F}$ if and only if~$S \in \mathcal{K}_{3,1}(G)$.
By Corollary~\ref{cor:fractionaltiling}, there exists a vertex weighting $\omega : V(\mathcal{F}) \rightarrow [0,1]$ such that 
\begin{align}
	\sum_{v \in F} \omega(v) &\ge 1 \textrm{ for all $ F \in \mathcal{K}_{3,1}(G)$,} &
	\min_{v \in V(G)} \omega(v) &= 0, &
	\sum_{v \in V(G)} \omega(v) &< n/3. \label{eqn:Farkas1}
\end{align}
Let $v_0 \in V(G)$ be such that $\omega(v_0) = 0$.
Let $W = N^+(v_0)$ and $e^{\pm} = e(G^{\pm}[W])$, so $|W| = \delta n$.
Note that, for each $xy \in E(G^{\pm}[W])$, $v_0xy \in \mathcal{F}$ and so $\omega(x)+\omega(y) \ge 1$ by~\eqref{eqn:Farkas1}.
Let $m$ be the size of largest matching in $G^{\pm}[W]$.
Hence, together with Theorem~\ref{thm:ErdosGallai}, we have 
\begin{align}
	\sum_{w \in W} \omega(w) \ge m \ge \min\left\{  \sqrt{e^{\pm}/2}-\gamma^5n,  \delta n - \sqrt{\delta ^2 n^2 - 2e^{\pm}} \right\}. \label{eqn:matching}
\end{align}
Again by~\eqref{eqn:Farkas1}, we deduce that $m \le n/3$ and so $e^{\pm} \le \binom{2n/3+1}{2} = 2n^2/9 + n/3$ 
by~Theorem~\ref{thm:ErdosGallai}. 
Note that  
\begin{align}
	\frac{2n^2}{9} + \frac{n}{3}\ge e^{\pm} & \ge e(G[W]) - \binom{\delta n}{2}
	= ( \delta n |W| - e(W,\overline{W}) )  - \binom{\delta n}{2}
	\ge \frac{\delta^2 n^2}{2} - e(W,\overline{W}).  \label{eqn:e(G+-)}
\end{align}
Together with our choice of~$\delta$, we have 
\begin{align*}
e(W,\overline{W}) \ge \frac{\delta^2 n^2}{2}  - \frac{2n^2}{9} - \frac{n}{3} \ge (1- \delta)(4/3- \delta)n^2.
\end{align*}
Let $s$ be the largest integer such that 
\begin{align}
		\sum_{w' \in \overline{W}} d^-(w') \ge  e(W,\overline{W}) \ge s \delta n  + (( 1-\delta)n - s )(4/3- \delta)n. \label{eqn:sumW}
\end{align}
In particular, we have
\begin{align}
e(W,\overline{W}) &< (s+1) \delta n + ( (1- \delta) n - (s+1) ) ( 4/3 - \delta)n \nonumber \\
    &= s \delta n + ( (1- \delta) n - s ) ( 4/3 - \delta)n  + 2(\delta -2/3)n.
    \label{eqn:uppere(W,W)}
\end{align}
We now bound $\sum_{w' \in \overline{W}} \omega(w')$ in the following claim. 

\begin{claim} \label{clm:weightbydegree}
We have
\begin{align*}
 \sum_{w' \in \overline{W}} \omega(w') \ge \left( 1- \frac{2}{3 \delta} \right)s.
\end{align*}
\end{claim}

\begin{proofclaim}
For $ v \in V(G)$, let $f(v) = \min \{ d^-(v)/n + \delta-1, \delta /2 \} $.
Consider $v \in V(G)$. 
For each $x \in N^+(v)$, 
\begin{align*}
	|N^+(x) \cap N^-(v) | \ge d^+(x) + d^-(v) - n \ge d^-(v) - (1- \delta)n.
\end{align*}
For each $y \in N^+(x) \cap N^-(v)$, $G[vxy]$ contains a~$C_3$, that is, $vxy \in E(\mathcal{F})$. 
Thus, in the link graph~$\mathcal{F}_v$, each $x \in N^+(v)$ has degree at least $d^-(v) - (1- \delta)n$. 
Hence (by a greedy argument) $\mathcal{F}_v$ contains a matching of size at least~$\min \{ d^-(v) - (1- \delta)n, \delta n /2 \}  =  f(v) n$. 
Let $M = \{ x_iy_i : i \in [f(v) n ] \}$ be a matching in~$\mathcal{F}_v$. 
By~\eqref{eqn:Farkas1}, we have
\begin{align*}
	 n/3 & > \sum_{v' \in V(G)} \omega(v') \ge  \sum_{i \in [f(v)n ] }\left(  \omega(x_i)+ \omega(y_i) \right) 
	\ge \sum_{i \in [f(v)n] } ( 1 - \omega(v)) =  ( 1 - \omega(v)) f(v) n.
\end{align*}
After rearranging, we obtain that for all $ v \in V(G)$, $\omega(v) > 1 - \frac{1}{3 f(v)}$ and so
\begin{align*}
	\omega(v) \ge
	\begin{cases}
		1- \frac{2}{3 \delta} & \text{if $d^-(v) \ge (1-\delta/2)n$,}\\
		1-\frac{1}{3(d^-(v)/n + \delta-1)} & \text{if $(4/3 - \delta)n < d^-(v) < (1-\delta/2)n$,}\\
		0 & \text{if $d^-(v) \le (4/3 - \delta)n$.}
	\end{cases}
\end{align*}
In particular, $1 - \frac{1}{3 f(v)}$ is a concave function in terms of~$d^-(v)$ when $d^-(v)>  (4/3 - \delta)n$.
Together with \eqref{eqn:sumW}, we get
\begin{align*}
	\sum_{w' \in \overline{W}} \omega(w') 
	\ge \left( 1- \frac{2}{3 \delta}\right) s 
\end{align*}
as required.
\end{proofclaim}

Let $\sigma = s/n$, so $0 \le \sigma \le 1-\delta$.
Note that 
\begin{align*}
\frac1n \left(\sqrt{\frac{e^{\pm}}2} - \gamma^5 + \left( 1- \frac{2}{3 \delta} \right)s \right)
&\overset{\mathclap{\text{\eqref{eqn:e(G+-)}}}}{\ge} \sqrt{\frac{\delta^2}4-\frac{e(W,\overline{W})}{2n^2}} - \frac{\gamma^5}{n} + \left( 1- \frac{2}{3 \delta} \right)\sigma\\
& \overset{\mathclap{\text{\eqref{eqn:uppere(W,W)}}}}{\ge}
\sqrt{\frac{\delta^2}4 - \frac{\sigma \delta   + ( 1-\delta - \sigma )(4/3- \delta)}2 - \frac{\delta -2/3}{n}} - \frac{\gamma^5}{n}  + \left( 1- \frac{2}{3 \delta} \right)\sigma 
\\
&\ge 1/3.
\end{align*}
Similarly, we have 
\begin{align*}
\frac1n \left( \delta n - \sqrt{\delta ^2 n^2 - 2e^{\pm}} + \left( 1- \frac{2}{3 \delta} \right)s \right)
>1/3 .
\end{align*}
Finally, by~\eqref{eqn:matching} and Claim~\ref{clm:weightbydegree}, we have
\begin{align*}
\sum_{v \in V(G)} \omega(v) &\ge  \sum_{w \in W} \omega(w) + 	\sum_{w' \in \overline{W}}\omega(w') 
\\
& \ge \min\left\{  \sqrt{e^{\pm}/2}- \gamma^5n,  \delta n - \sqrt{\delta ^2 n^2 - 2e^{\pm}} \right\} + \left( 1- \frac{2}{3 \delta} \right)s
\ge n/3,
\end{align*}
 contradicting~\eqref{eqn:Farkas1}.
\end{proof}




\color{black}

\section{Proof of Theorem~\ref{thm:main}} \label{sec:proofmain} 

We are now ready to combine our results on absorption and fractional tilings to prove \Cref{thm:main}.

\begin{proof}[Proof of~\Cref{thm:main}]
    Let $1/n_0 \ll 1/t \ll \phi \ll \eta \ll \eps, 1/r $ and suppose $G$ is an edge-coloured graph on $n\geq n_0$ vertices satisfying the conditions of the statement. Define \begin{align*}
\delta_r = 
	\begin{cases}
	5/6  & \text{if $r = 3$,}\\
	\frac{2+\sqrt{13/2}}{5}  & \text{if $r =4$,}\\
	1- \frac{1}{(2r-3)r} & \text{if $r \ge 5$.}
	\end{cases}
\end{align*}
By assumption we have $\delta^c(G)\geq (\delta_r +\eps)n \geq (1-\frac{1}{2(r-1)} + \eps)n$. By \Cref{prop:1-closed}, $G$ is $(1,\eta;K_r)$-closed. Hence we can apply \Cref{lemma:LoMarkstrom} with $s=1$ to find an absorbing set $W \subseteq V(G)$ of order at most~$\eta n$ such that both $G[U \cup W]$ and $G[W]$ contain perfect rainbow-$K_r$-tilings, for every $U \subseteq V(G) \setminus W$ satisfying $|U| \in r\mathbb{N}$ and $|U|\leq \phi n$. 

Define $G' = G[V(G)\setminus W]$. Then $\delta^c(G') \geq \delta^c(G) - |W| \geq (\delta_r + 4\eta)n$. 
By Lemmas~\ref{lma:fractionalrainbowK_4r},~\ref{lma:fractionalrainbowK_4}~and~\ref{lma:fractionalrainbowK_3}, every digraph $R$ on $t$ vertices with $\delta^+(G) \ge (\delta_r + \eta)t$ contains a perfect fractional $(r,r-2)$-tiling. Applying~\Cref{lem:almost_tiling} to $G'$ (with $4\eta, \phi^{2}/9$ and $\delta_r +\eta$ playing the roles of $\gamma, \eps$ and $\delta$ respectively), there exists a rainbow-$K_r$-tiling $\mathcal{T}_1$ in $G'$ covering all but at most $\phi n$ vertices. Let $U = V(G') \setminus V(\mathcal{T}_1)$, so that $|U| \in r\mathbb{N}$ and $|U| \le \phi n$. By our choice of $W$, we know that $G[U \cup W]$ contains a perfect rainbow-$K_r$-tiling $\mathcal{T}_2$. Therefore $\mathcal{T}_1 \cup \mathcal{T}_2$ is a perfect rainbow-$K_r$-tiling in $G$.
\end{proof}


\section{Remarks} \label{sec:remark}

We would like to know the minimum colour degree threshold for rainbow-$K_r$-tilings. 
We conjecture that the construction given by Proposition~\ref{prop:extremalconstruction} is sharp. 

\begin{conjecture} \label{conj}
Let $1/n \ll 1/r$ with $ n \equiv 0 \pmod r$. 
Let $G$ be an edge-coloured graph on $n$ vertices with $\delta^c(G) \ge \left( 1- \frac{r-1}{r^2-2} \right)n$. 
Then $G$ contains a perfect rainbow-$K_r$-tiling. 
\end{conjecture}

When $r=3$ we can prove that an absorption lemma already exists when the minimum colour degree is at least $(2/3 + \gamma)n$, which is much less than the conjectured bound. 
See Lemma~\ref{lem:absorbing} in Appendix.
Therefore in order to prove Conjecture~\ref{conj} for $r=3$, it suffices to find an almost perfect rainbow-$K_3$-tiling (or a perfect fractional $(3,1)$-tiling in the digraph setting). Note that our proof of~\Cref{lma:fractionalrainbowK_3} still works with $\delta^c(G)\geq (0.832+\gamma)n$, a minor improvement on $5/6$. 
However we believe this can be lowered further, but new ideas would be needed. 

For $r \ge 5$, we are unable to obtain any `reasonable' upper bound for the existence of a rainbow-$K_r$-tiling. 
One reason is that we are unaware of a density version of the Hajnal--Szemeredi theorem, that is, the Tur\'an number for $t$ vertex-disjoint~$K_{r-1}$.

In this paper, we only focus on rainbow-$K_r$-tilings.
A rainbow-$K_r$-tiling is a properly coloured $K_r$-tiling, so one could ask the minimum colour degree threshold of a properly coloured-$K_r$-tiling instead. 
However, we feel that it should be the same asymptotically as the one for rainbow-$K_r$-tilings.
Using Proposition~\ref{prop:conversion}, one can consider the corresponding digraph problem.
By the regularity lemma and Corollary~\ref{cor:conversionK_t}, finding a properly coloured $K_r$ is the same as finding a properly coloured~$K_r(r^3)$, which contains a copy of rainbow~$K_r$ by Lemma~\ref{lma:KMSV}.

\subsection{Generalising $(r,r-2)$-tilings.}

We show that one can translate a minimum colour-degree problem into a digraph problem. 
Recall that $\mathcal{K}_{r,r-2}$ is the set of all digraphs on $r$ vertices with complete base graph and the minimum out-degree at least $r-2$. 
We believe the thresholds for the existence for a perfect $(r,r-2)$-tilings should be asymptotically the same as for the existence for a perfect rainbow-$K_r$-tilings.
This is true when $r =3$ by Lemmas~\ref{lem:almost_tiling} and~\ref{lem:absorbing}.

Two natural extensions would be to relax the minimum out-degree and to remove the condition that the base graph is complete. 
Let $s < r $. 
Let $\mathcal{K}_{r,s}$ be the set of all digraphs on $r$ vertices with minimum out-degree at least~$s$ with complete base graph. 
Let $\mathcal{K}^*_{r,s}$ be the set of all digraphs on $r$ vertices with minimum out-degree at least~$s$.
Note that $\mathcal{K}_{r,r-1} = \mathcal{K}^*_{r,r-1}$ is the complete digraph on $r$ vertices. 

\begin{question}
Let $ s< r$. 
What is the minimum out-degree threshold that guarantees a perfect $\mathcal{K}_{r,s}$-tiling (or a perfect $\mathcal{K}^*_{r,s}$-tiling)?
\end{question}

\section*{Acknowledgment}
This project was initially funded through the London Mathematical Society undergraduate research bursary scheme, grant number URB-2022-13.

\bibliographystyle{amsplain}
\bibliography{RainbowTriangleTiling}

\appendix

\section{Absorption lemma for rainbow-$K_3$-tilings - revisit} \label{sec:absorbingrevisit}

In this section, we show that $\delta^c(G) \ge (2/3+\eps)n$ already guarantees an absorption lemma for the rainbow-$K_3$-tiling.
Note that $\delta^c(G) \ge (2/3+\eps)n$ is well below the conjectured extremal construction of $5n/7$.

\begin{lemma} \label{lem:absorbing}
Let $1/n \ll \phi \ll \eps \ll 1$.
Let $G$ be an edge-coloured graph on $n$ vertices with $\delta^c(G)\ge (2/3+ \eps) n$.
Then there exists an absorbing set $W\subseteq V(G)$ of order at most $\eps n$ so that for every $U \subseteq V(H) \setminus W$ with $|U| \le \phi n$ and $|U|\in 3\mathbb{N}$, both $G[W]$ and $G[U\cup W]$ contain perfect rainbow-$K_3$-tilings.
\end{lemma}

By Lemma~\ref{lemma:LoMarkstrom}, it suffices to prove the following. 

\begin{lemma} \label{lem:closed}
Let $1/n \ll \eta \ll 1/s \ll \eps \ll 1$.
Let $G$ be an edge-coloured graph on $n$ vertices with $\delta^c(G)\ge (2/3+ 3\eps) n$.
Then $G$ is $(s,\eta;K_3)$-closed.
\end{lemma}

Thus our goal for this section is to show that every edge-coloured graph~$G$ on $n$ vertices with~$\delta^c(G) \ge (2/3 + 3\eps) n$ is $(s,\eta;K_3)$-closed. Since we are only considering rainbow-$K_3$-tilings in this section, in what follows we will simply say a graph is $(s,\eta)$-closed to mean it is $(s,\eta;K_3)$-closed.

We now sketch the proof of Lemma~\ref{lem:closed}.
Our proof has flavour of the lattice absoprtion of Mycroft and Keevash~\cite{KeevashMycroft2015} and of Han~\cite{Han2017}.
By Proposition~\ref{prop:conversion}, we may assume that $G$ is a digraph instead. 
Apply the regularity lemma, Lemma~\ref{lma:szemeredi}, and obtain a reduced graph~$R$ such that $\delta^+(R) \ge (2/3+\eps/2)|R|$. 
We then introduce a weaker notion of $(s,\eta)$-closed, which we call weakly connected (see Section~\ref{sec:weakconnected}). 
Intuitively, weakly connected only requires one connector between distinct vertices, and vertices in the connector can be repeated. 
After showing that the reduced graph is weakly connected (Lemma~\ref{lma:weaklyconnected}), we then use properties of regularity to show that there exists an almost spanning vertex set that is $(s',\eta')$-closed in the edge-coloured graph~$G$.
We conclude by showing that this implies the whole of~$G$ is $(s'',\eta'')$-closed.

\subsection{Weakly connectedness} \label{sec:weakconnected}

Let $G$ be a a digraph. 
Recall that $\mathcal{K}_{3,1}(G)$ is the set of all subdigraphs in $G$ on $3$ vertices with complete base graph and the minimum out-degree at least~$1$. 
We define the notion of \emph{weakly connectedness in $G$} as follows.\footnote{Note that connectedness is a property of edge-coloured graphs while weakly connectedness is a property of digraphs.}
A pair of vertices~$x$ and~$x'$ is \emph{weakly $s$-connected} if there exists a vertex multiset $W (x,x')$ such that $|W (x,x')| = 3s-1$ and both $\{x\} \cup W (x,x')$ and $\{x'\} \cup W (x,x')$ can be partitioned into $S_1, \dots, S_s$ with $G[S_i] \in \mathcal{K}_{3,1}(G)$ for all~$i \in [s]$. 
A subset~$U$ of~$V(G)$ is said to be \emph{weakly $s$-connected in~$G$} if every pair of vertices in $U$ are weakly $s$-connected to each other.
If $V(G)$ is weakly $s$-connected in $G$ then we simply say that \emph{$G$ is weakly $s$-connected}.

Here are some simple properties of weakly $s$-connectedness, for which we omit the proof.

\begin{fact} \label{fact:weaklyconnected}
Let $G$ be a digraph. 
Then the following holds:
\begin{enumerate}[label={\rm (\roman*)}]
	\item if $U, W \subseteq V(G)$ are weakly $s$-connected in~$G$ and $U \cap W \ne \emptyset$, then $U \cup W$ is weakly $2s$-connected in~$G$; 
	\label{itm:wc3}
	\item if $x,x',y,y',z,z' \in V(G)$ are such that $\delta^+(G[xyz]), \delta^+(G[x'y'z']) \ge 1$, $y$ and $y'$ are weakly $s$-connected and $z$ and $z'$ are weakly $s$-connected, then $x$ and $x'$ are  weakly $(2s+1)$-connected;\label{itm:wc4}
	\item if $U \subseteq V(G)$ is weakly $s$-connected in~$G$, $\mathcal{K}_{3,1}(G[U]) \ne \emptyset$, $x \in V(G)$ and $\delta^+(G[xuu']) \ge 1$ for some $u,u' \in U$, then $U \cup \{x\}$ is weakly $(3s+2)$-connected in~$G$. \label{itm:wc5}
\end{enumerate}
\end{fact}

The aim of this subsection is to show that if $\delta^+(G) > 2n/3$, then $G$ is weakly connected.

\begin{lemma} \label{lma:weaklyconnected}
Let $1/s \ll 1/n$.
Let $G$ be a digraph on $n$ vertices with $\delta^+(G) > 2n/3 $.
Then $G$ is weakly $s$-connected. 
\end{lemma}

Note that $s$ is a function of~$n$. 
However, we will not track $s$ explicitly, and simply write weakly connected. 
This is because we will apply Lemma~\ref{lma:weaklyconnected} to the reduced graph which has a constant number of vertices. 
By Fact~\ref{fact:weaklyconnected}, weakly connected is an ``equivalent relationship''. 
We first show that if a vertex $x \in V(G)$ has large in-degree, then its neighbourhood is weakly connected. We will denote by $N^\pm(v)$ and $d^\pm(v)$ the neighbourhood and degree of $v$ in $G^\pm$ respectively.

\begin{proposition} \label{prop:weaklyconnected1}
Let $G$ be a digraph on $n$ vertices with $\delta^+(G) =\delta n $. 
Let $x \in V(G)$ with $d^-(x) > 2( 1 - \delta ) n $.
Then $N(x)$ is weakly connected. 
\end{proposition}

\begin{proof}
First we show that $N^+(x)$ is weakly connected. 
For $y,y' \in N^+(x)$, 
\begin{align*}
	|N^+(y) \cap N^+(y') \cap N^-(x) |  \ge d^+(y) + d^+(y')+d^-(x) - 2n  >0.
\end{align*}
Let $z \in N^+(y) \cap N^+(y') \cap N^-(x) $.
Each of $G[xyz]$ and $G[xy'z]$ contains a copy of $C_3$, implying that $y$ and $y'$ are  weakly connected. 
Thus $N^+(x)$ is weakly connected.

Note that $d^{\pm}(x) >  d^+(x) + d^-(x) - n  > (1- \delta) n$. 
For all $z \in  N^-(x)$, we have $N^+(z) \cap N^{\pm}(x) \ne\emptyset$.
Given any two $z, z' \in N^-(x)$, there exist $y,y' \in N^{\pm} (x)$ such that $x, y, z$ and $x, y', z'$ span a copy of~$J_3$.
Recall that $y$ and $y'$ is  weakly connected (as $y,y' \in  N^{\pm} (x) \subseteq N^+(x)$), so $z$ and $z'$ is  weakly connected by Fact~\ref{fact:weaklyconnected}\ref{itm:wc4}.
Hence $N^-(x)$ is  weakly connected. 
Since $N^-(x) \cap N^+(x) = N^{\pm}(x) \ne \emptyset$, the result follows by Fact~\ref{fact:weaklyconnected}\ref{itm:wc3}.
\end{proof}

\begin{proposition} \label{prop:weaklyconnected2}
Let $G$ be a digraph on $n$ vertices with $\delta^+(G) = \delta n $. 
Suppose that $X \subseteq V(G)$ is weakly connected and $\mathcal{K}_{3,1}(G[X]) \ne \emptyset$. 
Let $x \in X$ be such that $d^-(x) > (2 - \delta) n  -|X|$.
Then $N^+(x) \cup X$ is weakly connected. 
\end{proposition}

\begin{proof}
Consider any $y \in N^+(x)$. 
Note that 
\begin{align*}
 |N^+(y) \cap N^-(x) \cap X| \ge \delta^+(G) + d^-(x) + |X| - 2n > 0.
\end{align*}
For each $x' \in N^+(y) \cap N^-(x) \cap X$, $G[xx' y]$ forms a $C_3$. Fact~\ref{fact:weaklyconnected}\ref{itm:wc5} implies $\{y\} \cup X$ is weakly connected. Since this holds for all $y \in N^+(x)$, we are done by Fact~\ref{fact:weaklyconnected}\ref{itm:wc3}.
\end{proof}

\begin{proposition} \label{prop:connected}
Let $G$ be a digraph on $n$ vertices with $\delta^+(G) =\delta n $.
Suppose that $X \subseteq V(G)$ is weakly connected and $\mathcal{K}_{3,1}(G[X]) \ne \emptyset$. 
Let $z \in V(G)$ be such that $d^{+}(z, X) \ge 2 (1-\delta) n $.
Then $\{z\} \cup X$ is weakly connected. 
\end{proposition}

\begin{proof}
Let $ Z \subseteq N^{+}(z, X)$ of size $ |Z| = 2 (1-\delta) n$. 
Note that 
\begin{align*}
	\delta^+(G[Z]) \ge \delta^+(G) + |Z| - n \ge (1-\delta) n = |Z|/2.
\end{align*}
Hence there exists an edge~$uv$ in~$G^{\pm}[Z]$.
Note that $\delta^+(G[zuv]) \ge 1$.
Hence $\{z\} \cup X$ is weakly connected by Fact~\ref{fact:weaklyconnected}\ref{itm:wc5}. 
\end{proof}

We now prove Lemma~\ref{lma:weaklyconnected}.

\begin{proof}[Proof of Lemma~\ref{lma:weaklyconnected}]
Let $\delta^+(G) = \delta n$ and by deleting edges if necessary, we may assume that $d^+(v) = \delta n$ for all $v \in V(G)$. 
Let $x_0 \in V(G)$ be such that $d^-(x_0) = \Delta^-(G) \ge \delta^+(G) = \delta n > 2(1- \delta) n$. 
By Proposition~\ref{prop:weaklyconnected1}, $N(x_0)$ is weakly connected.

Let $X$ be the set of vertices that are weakly connected to~$N(x_0)$. 
We may assume that $|X| < n$ or else we are done. 
Since $|X|\geq |N(x_0)| \ge \delta n$, we have $\delta^+(G[X]) \ge \delta n - (n - |X|) > |X|/2$.
By Corollary~\ref{cor:directedK_3}, $\mathcal{K}_{3,1}(G[X]) \ne \emptyset$. 

Let 
\begin{align*}
	Y  & = V(G) \setminus X, &
	W & = X \cap N^-(Y), \\
	Z & = ( X \cap N^+(Y) ) \setminus W, &
	T & = X \setminus ( W \cup Z) =X \setminus N(Y). 
\end{align*}
Note that
\begin{align}
	|X| &\ge \delta n \text{ and }
	|T|, |Y| \le (1- \delta) n < (2\delta -1 ) n. \label{eqn:basic}
\end{align}
By Proposition~\ref{prop:connected}, we have $\delta n  - |Y| \le d^{+}(y, X) \le 2 (1-\delta) n $ for all $y \in Y$. 
Hence
\begin{align}
	|Y| \ge (3 \delta -2 ) n .\label{eqn:|Y|}
\end{align}

We now bound the $d^-(v)$ from above. 

\begin{claim} \label{clm:YWZT}
For $ v \in V(G)$, 
\begin{align*}
	d^-(v) \le 
	\begin{cases}
		\min \{ |Y|+|W|, (1- \delta)n + |Y|\}	& \text{if $v \in Y$},\\
		(1- \delta)n + |Y|	& \text{if $v \in W$},\\
		2(1 - \delta) n & \text{if $v \in Z$}, \\
		|X|	& \text{if $v \in T$}.
	\end{cases}
\end{align*}
\end{claim}

\begin{proofclaim}
Consider $y \in Y$. 
By our construction, $N^-(y) \subseteq Y \cup W$, so $d^-(y) \le |Y|+|W|$.
Suppose that  $d^-(y) > (1- \delta)n + |Y|$.
Consider $x \in N^+(y) \cap X$.
Note that 
\begin{align*}
  |N^-(y) \cap N^+(x) \cap X| \ge d^-(y) + \delta^+(G) +  |X| - 2n > 0.
\end{align*}
There exists $x' \in N^-(y) \cap N^+(x) \cap X$ implying that $xx'y$ spans a~$C_3$.
By Fact~\ref{fact:weaklyconnected}\ref{itm:wc5}, $y \in X$, a contradiction. 

Let $w \in W$ and suppose to the contrary that $d^-(w) > (1- \delta) + |Y|$.
Let $y \in N^+(w) \cap Y$, which exists as $w \in W$. 
Note that 
\begin{align*}
  |N^+(y) \cap N^-(w) \cap X| \ge \delta^+(G) + d^-(w) + |X| - 2n > 0.
\end{align*}
For each $x \in N^+(y) \cap N^-(w) \cap X$, $G[w x y]$ contains a $C_3$ implying that $y$ is weakly connected to~$X$ by Fact~\ref{fact:weaklyconnected}\ref{itm:wc5}, a contradiction.

If  there exists $z \in Z$ with $d^-(z) > 2(1 - \delta) n$, then Proposition~\ref{prop:weaklyconnected1} implies $N(z)$ is weakly connected.
Note that $Y \cap N(z) \ne \emptyset$ as $z \in N^+(Y)$, and $N(z) \cap X \ne \emptyset$, implying that $N(z) \cup X$ is weakly connected by Fact~\ref{fact:weaklyconnected}\ref{itm:wc3}, contradicting the maximality of~$X$.

Finally, note that $N^-(T) \subseteq X$ by the definition of~$T$.
\end{proofclaim}

In the next claim, we bound $|W \cup Y|$ and $d^-(y)$ further for~$y \in Y$. 

\begin{claim} \label{clm:YWZT2}
We have  $|W\cup Y| < \delta n $ and for all $y \in Y$, $d^-(y) \le (1- \delta)n$.  
\end{claim}

\begin{proofclaim}
First suppose that $|W\cup Y| \ge \delta n $.
Recall by \eqref{eqn:basic} that $|Y| \le (1-\delta)n < \delta n/2$.
By Claim~\ref{clm:YWZT}, we have
\begin{align*}
	\delta n^2  & =  e(G) = \sum_{v \in V(G)} d^-(v) \le  |X| |T \cup Z| + ( (1-\delta)n + |Y|))|W \cup Y|\\
	& = (n - |Y|)(n - |W \cup Y|) + ( (1-\delta)n + |Y|) |W \cup Y| = (n-|Y|)n - (\delta n- 2|Y| )|W \cup Y| \\
	& \le (n-|Y|)n - (\delta n- 2|Y| )\delta n = (1- \delta^2)n^2 + (2 \delta -1 )n |Y|\\
	& \overset{\mathclap{\text{\eqref{eqn:basic}}}}{\le} (1- \delta^2)n^2 + (2 \delta -1 )(1-\delta)n^2 = 3\delta(1-\delta)n^2.
\end{align*}
After rearrangement, we obtain that $0 \le \delta (2-3 \delta) <0$, a contradiction. 
Hence  $|W\cup Y| < \delta n $.

Suppose to the contrary that there exists $y \in Y$ such that $d^-(y) > (1- \delta)n$.
Consider any $x \in N^+(y) \cap X$. 
There exists $v \in N^+(x) \cap N^-(y)$ and $vxy$ forms a $C_3$. 
Note that $v \in Y$ or else $y$ is weakly connected to~$X$ by Fact~\ref{fact:weaklyconnected}\ref{itm:wc5}, a contradiction. 
Therefore $d^+(x,Y) >0$ and $x\in W$, implying that $N^+(y) \subseteq W \cup Y $, which is a contradiction as $|W\cup Y| < \delta n$.
\end{proofclaim}
Next suppose that $|W\cup Y| \ge (1 - \delta) n $.
By Claims~\ref{clm:YWZT} and~\ref{clm:YWZT2}, we have
\begin{align}
	\delta n^2  & =  e(G) = \sum_{v \in V(G)} d^-(v) \le  |X| |T| + 2(1- \delta)n |Z| + ( (1- \delta)n + |Y| ) |W| + (1- \delta) n |Y| \nonumber \\
	& =  ( n - |Z \cup W \cup Y|) ( n - |Y|) 	 + 2(1- \delta)n |Z| + ( (1- \delta)n + |Y| ) |W| + (1- \delta) n |Y| \nonumber \\
	& = (n - |Y|)^2 +  (1- \delta) n |Y| -  ((2 \delta -1) n - |Y| ) |Z| - ( \delta n - 2|Y| ) |W|. \label{eqn:connect2}
\end{align}
Since $|Y| \le (1- \delta ) n$ by~\eqref{eqn:basic}, we have $ \delta n - 2|Y|  \ge (2 \delta -1) n - |Y| > 0 $.
Note that $|Z \cup W \cup Y| = |V(G) \setminus T| \ge \delta n$  by~\eqref{eqn:basic} and $|W\cup Y| \ge (1 - \delta) n $, so \eqref{eqn:connect2} is maximised when $|Z| = (2\delta -1)n$ and $|W| = (1- \delta) n - |Y|$. 
Thus, we have 
\begin{align*}
	\delta n^2  & \le (n - |Y|)^2 +  (1- \delta) n |Y| -  ((2 \delta -1) n - |Y| ) (2\delta -1)n - ( \delta n - 2|Y| ) ( (1- \delta) n - |Y|) ,\\
	|Y|^2 & \le \delta (2- 3 \delta )n^2 <0,	
\end{align*}
a contradiction. 
Therefore, we have 
\begin{align}
|W\cup Y| < (1 - \delta) n . \label{eqn:W+Y}
\end{align}

Consider $z \in Z$. 
Let $y \in Y \cap N^-(z)$ (which exists by definition). 
If $d^-(z,N^+(y) ) > (1-\delta) n $, then there exists $v \in N^+(y) \cap N^{\pm}(z)$ as
\begin{align*}
	|N^+(y) \cap N^{\pm}(z)| = |N^{+}(z) \cap (N^+(y) \cap N^-(z))| \ge d^+(z) + d^-(z,N^+(y) )- n >0.
\end{align*}
Note that $v,z \in Z \cup W \subseteq X$ and $\delta^+(G[vyz]) \ge 1$ implying that $y$ is weakly connected to~$X$, a contradiction. 
Hence $d^-(z, N^+(y) ) \le (1- \delta) n$.
Recall that $N^+(y) \cap T = \emptyset$.
Thus
\begin{align}
	d^-(z, V(G) \setminus T ) & \le d^-(z, N^+(y) ) +|V(G) \setminus (T \cup N^+(y))| 
	\le 
	(1- \delta) n + |V(G) \setminus (T \cup N^+(y))| \nonumber\\
	& \le 2(1- \delta) n - |T|.\label{eqn:d^-(z,noT)}
\end{align}

We now consider the edges in $G \setminus T$.
Recall that $N^+(Y) \cap T = \emptyset$ and $N^-(Y) \subseteq W \cup Y$.
By Claim~\ref{clm:YWZT} and~\eqref{eqn:d^-(z,noT)}, we have 
\begin{align*}
 0 & = \sum_{v \in V(G) \setminus T} \left( d^-(v, V(G) \setminus T)  -  d^+(v, V(G) \setminus T) \right) = \sum_{v \in Z \cup W \cup Y} \left( d^-(v, V(G) \setminus T)  -  d^+(v, V(G) \setminus T) \right)\\
	&\le \Big( 2(1- \delta) n - |T| - (\delta n - |T|) \Big) |Z| 
			+ \Big((1-\delta) n +|Y| - (\delta n - |T|) \Big) |W|
			+ \Big( |W \cup Y| - \delta n \Big) |Y|\\
	& = - (3 \delta - 2)n |Z| - ( (2\delta - 1) n - |T|)|W| + ( |W|+|W \cup Y|- \delta n  )   |Y| 
	\\
	&\overset{\mathclap{\text{\eqref{eqn:basic}}}}{<} 0+0+( 2|W \cup Y|- \delta n  ) |Y|
	\overset{\mathclap{\text{\eqref{eqn:W+Y}}}}{<}  - (3 \delta -2) n |Y| <0,
\end{align*}	
a contradiction. 
\end{proof}


\subsection{Proof of Lemma~\ref{lem:closed}}

In this subsection, $G$ is an edge-coloured graph. 
The next lemma shows that it suffices to find an almost spanning $(s, \eta)$-closed vertex subset in~$G$. 

\begin{lemma} \label{lem:almostclosed}
Let $1/n \ll \eta' \ll \eta, 1/s, \eps$.
Let $G$ be an edge-coloured graph on $n$ vertices with $\delta^c(G) \ge (2/3 + \eps) n$. 
Suppose that $X \subseteq V(G)$ is $(s, \eta)$-closed with $|X| \ge (1- \eps) n$.
Then $G$ is $(2s+1,\eta')$-closed.
\end{lemma}

\begin{proof}
We first show that for each vertex $z \in V(G)$, there are many rainbow~$K_3$ in~$G[X \cup \{z\}]$ containing~$z$. 

\begin{claim} \label{claim:1}
For each $z \in V(G)$, there exist at least $\eps  n^2 /6$ pairs of vertices $x,x' \in X$ such that $G[zxx']$ is a rainbow~$K_3$.
\end{claim}

\begin{proofclaim}
Let $z \in V(G)$. Note that $\delta^c(G) + |X| - n \ge 2n/3$. Let $Z$ be a subset of $X \cap N(z)$ such that $|Z| = 2n/3$ and $c(zz')$ is distinct for all $ z' \in Z$.
Let $H$ be the digraph on~$Z$ such that $uv \in E(H)$ if and only if $uv \in E(G)$ and $c(uv) \ne c(zu)$.
Note that 
\begin{align*}
	\delta^+(H) \ge \delta^c(G) -1 + |Z| - n \ge (1/3 + \eps) n -1  \ge |Z|/2+ \eps n/4 .
\end{align*}
Hence the number of double edges in~$H$ is at least $\eps |Z| n /4 = \eps  n^2 /6$.
For each double edge~$xx'$ in~$H$, the vertices $z,x,x'$ span a copy of a rainbow~$K_3$ in $G$, as $c(zx) \ne c(xx') \ne c(zx')$ and $x,x' \in Z$. 
\end{proofclaim}

Consider distinct $z_1,z_2 \in V(G)$.
By Claim~\ref{claim:1}, there exist $\eps^2  n^4 /36$ many vertices $x_1,x'_1, x_2,x'_2 \in X$ such that $G[z_ix_ix'_i]$ is a rainbow~$K_3$ for $i \in [2]$. 
Since $X$ is $(s, \eta)$-closed, there exist at least $ (\eta n^{3s-1}/2)^2$ pairs $(S,S')$ such that $S,S', \{x_1,x'_1, x_2,x'_2\}$ are pairwise disjoint and $S$ and $S'$ are an $(x_1,x_2)$- and $(x'_1,x'_2)$-connector of length~$s$, respectively.
Then $S \cup S' \cup \{x_1,x'_1, x_2,x'_2\}$ is a $(z_1,z_2)$-connector of length~$2s+1$.
Hence, there are at least $ \frac{\eps^2  n^4}{36}\cdot\left(\frac{\eta n^{3s-1}}2\right)^2 \cdot \frac{1}{(2s+1)!}$ such connectors. 
The lemma follows as $\eta' \ll \eta, 1/s, \eps$.
\end{proof}

Recall that Proposition~\ref{prop:conversion} converts an edge-coloured graph~$G$ into a digraph~$H$.
The next corollary shows that most elements of $\mathcal{K}_{3,1}(H)$ correspond to a rainbow~$K_3$ in~$G$.

\begin{corollary} \label{cor:conversionK_3}
Let $G$ and $H$ be as defined in Proposition~\ref{prop:conversion}.
Then there are at most $3n^2$ many $S \in \binom{V(H)}{3}$ such that $H[S] \in \mathcal{K}_{3,1}(H)$ but $G[S]$ is not a rainbow~$K_3$.
Moreover, for each $v \in V(H)$, there are at most $3n^{3/2}$ many $S \in \binom{V(H)}{3}$ such that $ v \in S$ and $H[S] \in \mathcal{K}_{3,1}(H)$ but $G[S]$ is not a rainbow~$K_3$.
\end{corollary}

\begin{proof}
Let $S = \{x,y,z\}$ be such that $H[S] \in \mathcal{K}_{3,1}(H)$ but $G[S]$ is not a rainbow~$K_3$.
Without loss of generality, we have $c(xy) = c(xz)$. 
By the definition of $H$, we may further assume that $xy$ is a double edge in $H$ and $zx \in E(H)$ but $xz \notin E(H)$. 
For each $z \in N^-(x)$, there is at most one such~$y$. 
Hence the number of $\{y,z\}$ such that $\delta^+(H[xyz]) \ge 1$ but $c(xy) = c(xz)$ is at most $d^-_H(x) \le n$.
So there are at most $3n^2$ such~$S$. 

Now suppose that $S$ contains a fixed vertex~$v$. 
If $v = x$, then the argument above implies that there are at most $d^-_H(v) \le n$ such~$S$. 
If $v = z$, then $S$ is uniquely defined once $x$ is chosen and so there are at most $n-1$ such~$S$. 
Finally if $v = y$, then there are $n$ choices for $x$ and $\sqrt{n}$ choices for~$z \in N^-_H(x)$ with $c(xz) = c(xy)$. 
Hence there are at most $ n +n+ n^{3/2} \le 3n^{3/2}$ such~$S$. 
\end{proof}

We are now ready to prove Lemma~\ref{lem:closed}.

\begin{proof}[Proof of Lemma~\ref{lem:closed}]
Let $1/n \ll \eta \ll 1/s, \eta'_R \ll \eta_R \ll 1/N \ll d,\eps' \ll \eps \ll 1 $ and let~$s_R=(s-1)/4$.

By Proposition~\ref{prop:conversion} and Corollary~\ref{cor:conversionK_3}, there exists a digraph~$H$ on~$V(G)$ such that $\delta^+(H) \ge (2/3+2\eps) n $ and there are at most $3n^2$ many $S \in \binom{V(H)}{3}$ such that $H[S] \in \mathcal{K}_{3,1}(H)$ but $G[S]$ is not a rainbow~$K_3$.
Moreover, for each $v \in V(H)$, there are at most $3n^{3/2}$ many $S \in \binom{V(H)}{3}$ such that $ v \in S$ and $H[S] \in \mathcal{K}_{3,1}(H)$ but $G[S]$ is not a rainbow~$K_3$.

Apply Lemma~\ref{lma:szemeredi} (with $\eps',H$ playing the roles of~$\eps,G$) and obtain a spanning subgraph~$H'$ of~$H$ and an $(\eps', d,m,k)$-regular partition $\mathcal{Q}= \{V_0, V_1,  \dots, V_k\}$ of~$G'$ with $\eps^{-1} \leq k \leq N$ and $\Delta ( H - H') \le  (d + \eps')n$.
Thus $\delta^+(H') \ge (2/3+ \eps)n$.
Apply Lemma~\ref{lma:reducedgraph} (with $H',\eps',2/3+\eps$ playing the roles of $G,\eps,\delta^+$) and obtain an $(\eps', d)$-reduced digraph~$R$ respecting~$(H',\mathcal{Q})$ satisfying $\delta^+(R) > 2k/3$.
By Lemma~\ref{lma:weaklyconnected}, $R$ is weakly $s_R$-connected.

We now show that most vertices in $V_i$ are close to most vertices in $V_j$. 

\begin{claim}
For all distinct $i,j \in [k]$, 
there exist $V'_{i,\{i,j\}} \subseteq V_i$ and $V'_{j,\{i,j\}} \subseteq V_j$ such that $|V'_{i,\{i,j\}}|, |V'_{j,\{i,j\}}| \ge (1- 4 \eps') m$ and every $v_i \in V'_{i,\{i,j\}}$ and $v_j \in V'_{j,\{i,j\}}$ are $(s_R,\eta_R)$-close. 
\end{claim}

\begin{proofclaim}
Since $R$ is weakly $s_R$-connected, there is a vertex multiset $W (V_i,V_j) \subseteq V(R)$ of size~$|W (V_i,V_j)| = 3s_R-1$ such that, for $p \in \{i,j\}$, $V_p \cup W (V_i,V_j)$ can be partitioned into $S_{p,1}, \dots, S_{p,s_R}$ with $R[S_{p,q}] \in \mathcal{K}_{3,1}(R)$ for all~$q\in [s_R]$.
Define the auxiliary digraph~$J$ with $V(J) = V_i \cup V_j \cup W (V_i,V_j)$ and $E(J) = \bigcup_{p \in \{i,j\}} \bigcup_{q\in [s_R]} E(R[S_{p,q}])$.
Here we assume that the vertices of~$J$ are distinct, so $|V(J)| = 3s_R+1$, and if we say $V_i$ (or $V_j$), then we assume that $V_i \notin W (V_i,V_j)$ (or $V_j \notin W(V_i,V_j)$).
Thus $\Delta(J) \le 4$ and there is a natural homomorphism~$\varphi$ from $J$ to $R$. 

Apply Lemma~\ref{lma:counting} (with $J,V_i,V_j$ playing the roles of $H,u,u'$) and obtain
$V'_{i,\{i,j\}} \subseteq V_{i}$ and $V'_{j,\{i,j\}} \subseteq V_{j}$ with $|V'_{i,\{i,j\}} |, | V'_{ j,\{i,j\}} |\ge (1-12\eps)m $ such that, for every pair  $v_i \in V'_{i,\{i,j\}}$ and $v_j \in V'_{j,\{i,j\}}$, there are at least $2 \eta_R n^{|J| - 2}$ embeddings~$\varphi'$ of~$J$ in~$H$ such that $\varphi'(V_i)  \in V_i$, $\varphi'(V_j) = v_j$ and $\varphi'(U) \in U$ for $U \in W (V_i,V_j)$.

Consider $v_i \in V'_{i,\{i,j\}}$ and $v_j \in V'_{j,\{i,j\}}$.
If $\varphi' (W (V_i,V_j))$ is not a $(v_i,v_j)$-connector of length~$s_R$, then $G[S_{p,q}]$ is not a rainbow~$K_3$ for some $p \in \{i,j\}$ and~$q\in [s_R]$.
If $\{v_i, v_j\}  \cap \varphi(S_{p,q}) = \emptyset$, then there are at most $3n^2$ choices for $\varphi(S_{p,q})$, and $3n^{3/2}$ choices otherwise. 
Thus the number of embeddings~$\varphi$ such that $\varphi (W (V_i,V_j))$ is not a $(v_i,v_j)$-connector is at most 
\begin{align*}
2(s_R-1) \cdot 3n^2 \cdot n^{|J| -5} + 2 \cdot 3n^{3/2} \cdot n^{|J| -4} \le \eta_R n^{|J| - 2}.
\end{align*}
Thus $v_i$ and $v_j$ are $(s_R, \eta_R)$-close. 
\end{proofclaim}

Therefore $\bigcup_{i \in [k-1]} V'_{i,\{i,k\}}$ is $(2s_R,\eta_R')$-closed with
\begin{align*}
 \left| \bigcup_{i \in [k-1]} V'_{i,\{i,k\}} \right| \ge (1- 4 \eps') m (k-1) \ge (1-\eps )n . 
\end{align*}
Lemma~\ref{lem:almostclosed} implies that $G$ is $(4s_R+1, \eta)$-closed.
\end{proof}

\end{document}